\documentclass[a4paper,12pt]{article}

\usepackage[colorlinks=true, pdfstartview=FitV, linkcolor=black, citecolor=red, breaklinks=true, urlcolor=red]{hyperref}
\usepackage{amssymb,amsmath,amsthm}
\usepackage{graphicx,color,mathrsfs}

\newcommand{\beq}{\begin{equation}}
\newcommand{\eeq}{\end{equation}}

\usepackage[english]{babel}

\newtheorem{Theorem}{Theorem}

\newtheorem{lemma}[Theorem]{Lemma}

\newtheorem{remark}[Theorem]{Remark}

\newtheorem{corollary}[Theorem]{Corollary}
\newtheorem{defn}[Theorem]{Definition}

\newtheorem{proposition}[Theorem]{Proposition}

\newcommand{\bi}{\begin{itemize}}

\newcommand{\ei}{\end{itemize}}
\newcommand{\bd}{\begin{description}}
\newcommand{\ed}{\end{description}}
\newcommand{\be}{\begin{enumerate}}
\newcommand{\ee}{\end{enumerate}}

\newcommand{\bqn}{\begin{eqnarray}}
\newcommand{\eqn}{\end{eqnarray}}
\newcommand{\eqnn}{\nonumber\end{eqnarray}}
\newcommand{\eqnl}[1]{\label{#1}\end{eqnarray}}

\newcommand{\ba}[1]{\begin{array}{#1}}
\newcommand{\ea}{\end{array}}

\newcommand{\R}{\mathbb{R}}

\newcommand{\lam}{\lambda}
\newcommand{\al}{\alpha}

\newcommand{\vep}{\varepsilon}

\newcommand{\fixmeyc}[1]{\textcolor{red}{[Y:~#1]}}

\newcommand{\diag}{\operatorname{diag}}
\newcommand{\RR}{\mathbb{R}}
\newcommand{\NN}{\mathbb{N}}
\newcommand{\Id}{\operatorname{Id}}
\newcommand{\rf}{\mathbf{r}}

\usepackage[mathscr]{euscript}
\newcommand{\KL}{\mathscr{KL}}
\newcommand{\Ks}{\mathscr{K}}
\newcommand{\As}{\mathscr{A}}
\newcommand{\Ts}{\mathscr{T}}

\title{\LARGE \bf Stabilization for a perturbed chain of integrators
  in prescribed time}

\author{{Yacine Chitour, Rosane Ushirobira and Hassan Bouhemou}
\thanks{Y. Chitour is with
Laboratoire des Signaux et Syst\`emes (L2S, UMR CNRS 8506), CNRS - CentraleSupelec - Universit\'e Paris-Sud, 3, rue Joliot Curie, 91192, Gif-sur-Yvette, France, and Rosane Ushirobira is with Inria, Univ. Lille, CNRS, UMR 9189 - CRIStAL - Centre de Recherche en Informatique Signal et Automatique de Lille, F-59000 Lille, France.}
\thanks{\tt yacine.chitour@l2s.centralesupelec.fr, rosane.ushirobira@inria.fr, HBCHOUF@gmail.com }
 \footnote{This research was partially supported by the iCODE Institute, research project of the IDEX Paris-Saclay, and by the Hadamard Mathematics LabEx (LMH) through the grant number ANR-11-LABX-0056-LMH in the ``Programme des Investissements d'Avenir''.}
 }

\begin{document}
\maketitle

\begin{abstract} In this paper, we consider issues relative to prescribed time stabilisation of a chain of integrators of arbitrary length, either pure (i.e., where there is  no disturbance) or perturbed. In a first part, we revisit the proportional navigation feedback (PNF) approach and we show that it can be appropriately recasted within the framework of time-varying homogeneity.  As a first consequence, we first recover all previously obtained results on PNF with simpler arguments. We then apply sliding mode inspired feedbacks to achieve prescribed stabilisation with uniformly bounded gains. However,
 all these feedbacks are robust to matched uncertainties only.
 In a second part, we provide a feedback law yet inspired by sliding mode which not only stabilises the pure chain of integrators in
 prescribed time but also exhibits some robustness in the presence of measurement noise and unmatched uncertainties. 

\end{abstract}
\tableofcontents
\section{Introduction}
In this paper, we consider the following problem: for $n$ positive
integer and $T$ positive real number, the perturbed chain of integrators is the control system given by 
\beq\label{eq:PCofI-0} 
\dot x(t)=J_nx(t)+\left(d(t)+b(t)u(t)\right)
\ e_n,\qquad t\in[0,T), \quad x(t)\in \mathbb{R}^n,\quad
  u(t)\in\mathbb{R} 
\eeq 
where $(e_i)_{1\leq i\leq n}$ denotes the
  canonical basis of $\mathbb{R}^n$, $J_n$ denotes the $n$-th Jordan
  block (i.e., $J_ne_i=e_{i-1}$ for $1\leq i\leq n$ with the
  convention $e_0=0$), $d(\cdot)$ and $b(\cdot)$ denote respectively a matched uncertainty and the
  uncertainty on the control respectively. Moreover we assume that
  there exists $\underline{b}>0$ such that
  \beq\label{eq:b0} 
  b(t)\geq \underline{b}, \, \quad \forall t\in
            [0,T).  
            \eeq

Our goal consists in designing  a feedback control $u$ that renders the system
fixed-time input-to-state stable in any time $T >0$ (prescribed-time stabilisation) and possibly convergent to
zero (PT+ISS+C) (cf. \cite{Krstic2017}, and Definition~\ref{def:PT-ISS} given below). Note
also that one may asks robustness properties in the presence of noise measurement $d_1$ and unmatched uncertainty $d_2$, for instance, if the feedback control $u$ is static, it takes the form $u=F(x+d_1)+d_2$, where the feedback law $F(x)$ stabilises the $n$-th order pure chain of integrators.

Prescribed time stabilisation is a more difficult objective than mere finite time stabilisation and it has a long history especially in 
missile guidance \cite{Zarchan} and other applications. 
There are two main approaches for solving this problem: 
$(a)$ \emph{proportional navigation feedback} (PNF), which is linear in the state $x$ and with time-varying gains blowing up to infinity towards the prescribed fixed time; $(b)$ optimal control with a terminal constraint, where such a dependency of the gains is implicit. Stemming from the 
PNF design for second order chains of integrators, a general approach is proposed in \cite{Krstic2017} for $n$-th order perturbed chains of integrators, i.e., in \eqref{eq:PCofI-0}). The feedback law has the form $u(t)=K^TP_1(\frac1{T-t})(x(t))+P_2(\frac1{T-t})$, where $P_1$ and $P_2$  are polynomials (either matrix or real valued) and the vector $K$ must be chosen in a rather involved way.
The first term in the feedback is definitely of PNF type but the second one is only necessary for the convergence argument and does not appear for second order models for instance. Anyway, the controller in \cite{Krstic2017} does tend to zero as $t$ tends to $T$ even though the gains blow up and it exhibits excellent robustness 
properties in the case of matched uncertainties. However, the authors only suggest that it behaves poorly in case of measurement noise or unmatched uncertainties and also claim that all known techniques (including theirs) do not work in case of unmatched uncertainties. Finally, due to the rather complicated stability analysis as well as the involved construction of the feedback, it is not clear how to measure quantitatively the limitations of that feedback (see Section $3.2$ in \cite{Krstic2017}) and to possibly improve its results.

The first part of the present paper aims at revisiting 
PNF design with a new perspective. We show that it can be naturally seen as a particular instance of 
  weighted-homogeneous control systems (cf. \cite{SEFL} for instance) 
  with the usual homogeneity coefficient not anymore constant but being a time-varying 
  function. Indeed, recall that a PNF has the form 
  \begin{equation}\label{eq:PNF}
  u_{PNF}(t)=\sum_{i=1}^n\frac{k_ix_i(t)}{(T-t)^{n-i+1}},\quad t\in [0,T),
  \end{equation}
 where $K=(k_1,\cdots,k_n)\in\mathbb{R}^n$. Let $\rf=(n-i+1)_{1\leq i\leq n}$ the weight vector and set, for $\lambda>0$, $D^\rf_\lambda$ to be the diagonal matrix made of the $\lambda^{n-i+1}$'s. Rewriting $u_{PNF}(t)=K^TD^\rf_{\lambda(t)}x(t)$ with 
 $\lambda(t)=\frac1{T-t}$ suggests at once to consider the new state $y(t)=D^\rf_{\lambda(t)}x(t)$ and $u_{PNF}$ reduces to $K^Ty$. Now, the dynamics of $y$ with respect to the time scale $s(t)=\ln(\frac{T}{T-t})$ turns out to be
  \begin{equation}\label{eq:fresh}
 \frac{dy}{ds}=(D_{\rf}+J_n)y+(b(s)u(s)+d(s))e_n,\quad s\geq 0,
 \end{equation}
 where $D_{\rf}$ is the constant diagonal matrix  made of the $n-i+1$'s. Hence, the original problem of prescribed time stabilisation of \eqref{eq:PCofI-0} in time $T>0$ has been reduced to the stabilisation of an $n$-th order perturbed chain of integrators with the extra term $D_{\rf}y$ in 
  \eqref{eq:fresh} with respect to \eqref{eq:PCofI-0}.
 Note that, to the best of our knowledge, it seems to be the first time that one considers a time-varying homogeneity coefficient in the context of stabilisation of weighted-homogeneous systems in that general manner. Usually, the homogeneity coefficient, when non constant, is state-dependent (cf. \cite{praly1997} as the pioneering reference for fixed-time stabilisation of linear systems, then \cite{Moreno11} and \cite{hcl12} for instance, in the case of second order and $n$-th order perturbed chains of integrators respectively.)
  
  With the previous viewpoint, it is immediate to see that PNF (and its variant given in \cite{Krstic2017}) is nothing more but the 
  stabilisation of \eqref{eq:fresh} with a linear feedback. As a consequence, we recover all the results of \cite{Krstic2017} with much simpler arguments and the limitations (as well as the advantages) of such a feedback appear in a transparent way. In particular, 
  our convergence analysis easily reduces to the verification of an LMI (see Proposition~\ref{prop:LMI0} below), whose solution is 
  essentially given in \cite{CS-2010}. Moreover, the fact that measurement noise and unmatched uncertainties in \eqref{eq:PCofI-0} cannot be handled with that linear feedback is obvious since the corresponding disturbances become amplified by $D^\rf_{\lambda(t)}$ in \eqref{eq:fresh} and one can measure explicitly their destabilising effect.
  
  One can then turn to other types of stabilisation for \eqref{eq:fresh}. If the settling time of the system associated to some feedback law $u=F(s,y)$ is infinite (i.e., the supremum over the initial conditions $x_0$ of the time needed to reach the origin for the trajectory of 
  \eqref{eq:fresh} fed by $u=F(s,y)$ and starting at $x_0$), 
  then we will unavoidably face the numerical challenge of
  plotting $y(s(t))=D^\rf_{\lambda(t)}x(t))$ in \eqref{eq:PCofI-0}, with $D^\rf_{\lambda(t)}$ growing unbounded as $t$ tends to $T$. Therefore, we should aim at feedback laws $u=F(s,y)$ providing 
  fixed-time convergence for the $y$ variable. On the other hand,
 recall that the $n$-th order perturbed chain of integrators is the basic model for sliding mode control, cf. \cite{SEFL}, for which there exist plenty of efficient finite time stabilizers with eventually good robustness properties. 
  At the heart of these stabilizers, lies the technic of weigthed-homogeneity with \emph{constant} homogeneity coefficient.
  We will show that this technic easily extends to handle  \eqref{eq:fresh} and its extra linear term $D_{\rf}y$ to produce 
  fixed-time stabilizers for \eqref{eq:fresh} under the assumption that bounds on $b$ and $d$ in \eqref{eq:fresh} are known a priori. In particular, under that assumption, this resolves in a satisfactory manner one of the issues raised in \cite{Krstic2017}, namely that of avoiding a gain growing unbounded without sacrifice on the regulation accuracy in $x$. 
   
  The second objective of the paper consists in addressing the difficult issues of robustness with respect to measurement noise 
  and unmatched uncertainties for prescribed time stabilisation of 
  \eqref{eq:PCofI-0}. As mentioned earlier regarding time-varying homogeneity approach, the disturbances corresponding to these perturbations become, at the best, amplified by $D^\rf_{\lambda(t)}$ in \eqref{eq:fresh}. It is not clear at all how to handle \eqref{eq:fresh} with disturbances growing unbounded. This is why we present in the second part of the paper a feedback design
 that does not involve any time-varying function $\lam(\cdot)$. This will allow us to provide partial robustness results in case of measurement noise and unmatched disturbances on the feedback. Here, robustness must be understood in the ISS setting of Definition~\ref{def:ISpS} and not anymore according to Definitions~\ref{def:PT-ISS} and~\ref{def:PT-ISS-C}. 
Our construction is based on fixed-time stabilisation with a control on the settling time in the case of an unperturbed chain of
 integrators and then on the use a simple trick to extend that solution to prescribed-time stabilisation. To perform that strategy, one must get an explicit hold on several parameters. To be more precise, the fixed-time stabilisation design relies on sliding mode feedbacks with state-dependent homogeneity degree. This idea was first considered in \cite{hcl12} and \cite{HLCH-2017} with a completely explicit feedback law. The latter bears a serious drawback since it is discontinuous. This defect has been removed in a subsequent work in \cite{LREPP-2018}, relying on an appropriate perturbation argument. However that latter solution does not bear an explicit character, which is an issue to estimate the settling time, and hence it requires important extra work for practical implementations. Moreover, it can be adapted only to a restricted set of perturbations.

 Our feedback design for fixed-time stabilisation 
 relies on the sliding mode feedback
 laws proposed by \cite{Hong-2002} for finite-time stabilisation of an  $n$-th order pure chain of integrators . Recall that, in that reference, it is proved that, for
 every homogeneity parameter $\kappa\in [-\frac1n,\frac1n]$, there exists a
 control law $ u=\omega_\kappa^{H}(x)$ which stabilizes $\dot
 x=J_nx+u\ e_n $ and a Lyapunov function $V_{\kappa}$ for
 the closed-loop system satisfying $\dot V_{\kappa} \le -C
 V_{\kappa}^{\frac{2+2\kappa}{2 + \kappa}}$, for some positive
 constant $C$, independent of $\kappa$.  One of the main avantages of these feedbacks and Lyapunov functions is that they admit explicit closed forms formulas computable once the dimension $n$ is given. In order to first obtain fixed-time stabilisation, we choose, as in
 \cite{HLCH-2017}, a feedback law of the type $
 u=\omega_{\kappa(x)}^{H}(x)$, where the homogeneity parameter is a state 
 function and, by using the smart idea of \cite{LREPP-2018},
 we can also make $x\mapsto \kappa(x)$ continuous. We finally use a standard homogeneity trick to pass from fixed-time to prescribed-time stabilisation.

 The structure of the paper goes as follows. In Section
 \ref{sec:defns}, general stability notions and
 homogeneity properties are recalled. In Section \ref{sec:hom},
 time-varying weighted homogeneity is considered for $n$-th order perturbed chains of integrators: Subsection \ref{ss:linear} studies thoroughly linear time-varying homogeneous feedbacks
 while in Subsection \ref{ss:FTF}, we provide sliding mode based feedbacks with uniformly bounded gains. We gather in Section \ref{s:robust} a new design of a sliding mode inspired feedback for which we characterise explicitly the parameters and we give some ISS properties in presence of measurement noise and unmatched disturbances.
 Finally we collect in an appendix the proofs of technical results used in the text.

\section{Stability definitions} \label{sec:defns}

In this paper, we will consider various non autonomous differetial equations $\dot{x} = f(x,t)$, where $x\in\R^{n}$ and $f:\R^{n}\to\R^{n}$ is a vector field. When it exists, the solution of
$\dot{x} = f(x,t)$ for an initial condition $x_0\in\R^{n}$ is denoted by $X(t,x_0)$. We recall the main stability notions needed in the paper, 
\cite{Khalil-2002}.

\begin{defn}
 Let $\Omega$ be an open neighborhood of a forward
invariant set $\As \subset\R^{n}$.\footnote{Meaning that for $x_0 \in
  \As$ the solution $X(t,x_0) \in \As$ for all $t \geq 0$.} At
$\As$, the system is said to be:
\begin{itemize}
\item[(a)] {\em Lyapunov stable} if for any $x_{0}\in\Omega$ the
  solution $X(t,x_{0})$ is defined for all $t\geq0$, and for any
  $\epsilon>0$, there exists $\delta>0$ such that for any
  $x_{0}\in\Omega$, if $\Vert x_{0}\Vert_{\As}\leq\delta$ then $\Vert
  X(t,x_{0})\Vert_{\As}\leq\epsilon$, $\forall t \geq0$.

\item[(b)] {\em asymptotically stable} if it is Lyapunov stable and
  for any $\kappa>0$, $\epsilon>0$ there exists
  $T(\kappa,\epsilon)\geq0$ such that for any $x_{0}\in\Omega$, if
  $\Vert x_{0}\Vert_{\As}\leq\kappa$ then $\Vert
  X(t,x_{0})\Vert_{\As}\leq\epsilon$, $\forall t\geq
  T(\kappa,\epsilon)$.

\item[(c)] {\em finite-time converging from} $\Omega$ if for any
  $x_{0}\in\Omega$ there exists $0\leq T<+\infty$ such that
  $X(t,x_{0})\in\As$ for all $t\geq T$. The function
  $\Ts_{\As}(x_{0})=\inf\left\{T\geq0 \mid X(t,x_{0})\in\As,
  \;\forall t\geq T\right\}$ is called the\emph{ settling time for $x_0$} of the system.
  
\item[(d)] {\em finite-time stable} if it is Lyapunov stable and
  finite-time converging from $\Omega$.

\item[(e)] {\em fixed-time stable} if it is finite-time stable and
  $\sup_{x_{0}\in\Omega}\Ts_{\As}(x_{0})<+\infty$ and the latter is referred as the \emph{settling time} of the system.
.
\end{itemize}
\end{defn}

Furthermore, for prescribed-time stability and robustness issues, we consider disturbances $d:[0,\infty)\to\mathbb{R}^p$ which are
  measurable functions where $\Vert d\Vert_{[t_0,t_1)}$ denotes the
    essential supremum over any time interval $[t_0,t_1)$ in
      $[0,\infty)$. If $[t_0,t_1)=[0,\infty)$, we say that $d$ is
            bounded if $\Vert d\Vert_{\infty}:=\Vert
            d\Vert_{[0,\infty)}$ is finite. We have the following two
            definitions (cf.
            \cite{Krstic2017} and \cite{LREPP-2018}).
\begin{defn}\label{def:PT-ISS}
A system $\dot{x} = f(x,t,d)$ is {\em prescribed-time input-to-state
  stable in time $T$} (PT-ISS) if there exist functions $\beta \in
\KL$\footnote{A function $\gamma: \RR_+ \to \RR_+$ is said to belong
  to a class $\Ks$ if it is strictly increasing and continuous with
  $\gamma(0)=0$. A function $\alpha$ is said to belong to a class
  $\Ks_\infty$ if $\alpha \in \Ks$ and it increases to infinity. A
  function $\beta: \RR_+ \times \RR_+ \to \RR_+$ is said to belong to
  a class $\KL$ if for each fixed $t \in \RR_+$, $\beta(\cdot, t) \in
  \Ks_\infty$ and if for each fixed $s \in \RR_+$, $\beta(s,t)
  \underset{t \to \infty}{\longrightarrow} 0$.}, $\gamma \in \Ks$ and $\lambda:[t_0,t_0+T)\to\mathbb{R}^*_+$ such that $\lambda$ tends to infinity as $t$ tends to $t_0+T$ and,
for all $t \in [t_0, t_0+T)$ and bounded $d$, $|x(t)| \leq
  \beta \left( |x_0|, \lambda(t) \right) + \gamma
  \left(\lVert d \rVert_{[t_0,t]}\right)$.
\end{defn}

\begin{defn}\label{def:PT-ISS-C}
A system $\dot{x} = f(x,t,d)$ is {\em fixed-time input-to-state stable
  in time $T$ and convergent to zero} (PT-ISS-C) if there exist
functions $\beta$, $\beta_f \in \KL$, $\gamma \in \Ks$ and $\lambda:[t_0,t_0+T)\to\mathbb{R}^*_+$ such that $\lambda$ tends to infinity as $t$ tends to $t_0+T$ and,
for all $t \in [t_0, t_0+T)$ and bounded $d$,
$|x(t,d,x_0)| \leq \beta_f \left(\beta
  \left( |x_0|, t-t_0 \right) + \gamma \left(\lVert d
  \rVert_{[t_0,t]}\right),\lambda(t) \right)$.
\end{defn}

\begin{defn}\label{def:ISpS}
A system $\dot{x} = f(x,t,d)$ is {\em input-to-state practically stable
} (ISpS) if, for any bounded disturbance $d$, there exist functions $\beta\in \KL$, $\gamma\in \Ks$ and $c>0$ such that, for all $t\geq 0$ and bounded $d$,
$|x(t,d,x_0)| \leq \beta
  \left( |x_0|, t \right)+\gamma \left(\lVert d
  \rVert_{[0,t]}\right)+c$.
 The system is {\em input-to-state stable} (ISS) if $c=0$.
\end{defn}
Note that $(PT-ISS)$ is a much stronger property than ISS.
\begin{remark}
Definitions~\ref{def:PT-ISS} and \ref{def:PT-ISS-C} have been given in \cite{Krstic2017} but with the explicit choice $\lambda(t)=\frac{t-t_0}{T+t_0-t}$. 
\end{remark}
Next, basic definitions of homogeneity are collected.

\begin{defn}\hfill
  \begin{itemize}
    \item[(i)] A function $f : \RR^n \to \RR$ is said to be {\em
      homogeneous} of degree $m \in \RR$ with respect to the weights
      $\mathbf{r} = (r_1,...,r_n) \in \RR^n_{>0}$ if for every $x \in
      \RR^n $ and $\lambda \in \mathbb{R}_+^*$, $ f \left(
      D^{\mathbf{r}}_{\lambda} \ x \right) = \lambda^m f(x)$,
      where $D^{\mathbf{r}}_{\lambda} = \diag \left(
      \lambda^{r_i} \right)_{i=1}^n$ defines a family of
      dilations. We also say that $f$ is $\mathbf{r}$-homogeneous of
      degree $m$.

\item[(b)] A vector field $\Phi= (f_1, \dots, f_n) : \RR^n \to \RR^n$
  is said to be {\em homogeneous} of degree $m\in \RR$ if for all $1
  \leq k \leq n$, for all $x \in \RR^n $ and $\lambda \in
  \mathbb{R}_+^*$, $ f_k \left( D^{\mathbf{r}}_{\lambda} \ x \right) =
  \lambda^{m+r_k} f_k(x)$, i.e., each coordinate
  function $f_k$ is homogeneous of degree $m+r_k$. We also say that
  $F$ is $\mathbf{r}$-homogeneous of degree $m$.

\item[(c)] Let $\Phi$ be a continuous vector field. If $\Phi$ is
  $\mathbf{r}$-homogeneous of degree $m$, then the system $\dot{x} =
  \Phi(x), \ x \in \RR^n$ is {\em $\mathbf{r}$-homogeneous} of degree
  $m$.
  \end{itemize}
 \end{defn}

The next lemma is important in the proof of our results in Section
\ref{s:robust} (see for instance \cite{LREPP-2018}).

\begin{lemma} \cite{Nakamura} \label{le:notreL1}
Let $\dot x = f(x,t)$ be a $\rf$-homogenous system of degree $\kappa$
asymptotically stable at the origin. Then at the origin, the system is
globally finite-time stable if $\kappa < 0$, globally exponentially
stable if $\kappa = 0$ and globally fixed-time stable with respect to
any open set containing the origin
 if $\kappa > 0$.
\end{lemma}

%

\section{Time-varying homogeneity} \label{sec:hom}
Let $n$ be a positive integer, $(e_i)_{1\leq i\leq n}$
the canonical basis
of $\mathbb{R}^n$ and $J_n$ the $n$-th Jordan block, i.e. $J_ne_i=e_{i-1}$, $1\leq i\leq n$, with the convention that $e_0=0$. 
For $\lambda>0$, using the notation above for $D^\rf_\lambda$ (see also
\cite{CS-2010}), one has, for $r_i:= n-i+1$, $1\leq i\leq n$,
\begin{equation}\label{eq:dlam}
D^\rf_{\lambda}=\diag(\lambda^{r_i})_{i=1}^n, \, D^\rf_\lam J_n
\left(D^\rf_\lam\right)^{-1}=\lam J_n,\, D^\rf_\lam e_n=\lam e_n. \eeq
(The second relation above simply says that the linear vector field induced by $J_n$ on $\R^n$ is $\mathbf{r}$-homogeneous of degree $-1$.)

In the literature devoted to prescribed time stabilization
(see \cite{Krstic2017} and references therein) and as clearly stated in
Definitions~\ref{def:PT-ISS} and \ref{def:PT-ISS-C}, the quantity
$\frac{t-t_0}{T+t_0-t}$ is a new time scale which
tends to infinity as $t$ tends to the prescribed convergence time
$T$. This fact suggests to consider the homogeneity parameter $\lambda$
depending on the time $t$ in such a way that, if one sets the new time to be equal to
\begin{equation}\label{eq:newT}
s:[0,T)\to \mathbb{R}_+^*, \quad s(t)=\int_0^t\lam(\xi)d\xi, \eeq then $s(t)$ tends to infinity as $t$ tends to $T$. In that case,
  it is natural to consider the change of coordinates and time given by
\begin{equation}\label{eq:COC}
y(s)=D^\rf_{\lam(t)}x(t), \ \forall \ t\in[0,T).
\end{equation}
  
In order to analyse the dynamics of $y$ in the new time $s$, we use $y'$ to denote the derivative of $y$ with respect to. Using \eqref{eq:PCofI-0},
\eqref{eq:dlam} and \eqref{eq:newT}, we obtain: 
\beq \label{eq:der0}
\lam \ y'=\dot y =\dot \lam \ \frac{\partial
  D^\rf_{\lam}}{\partial\lam}(D^\rf)^{-1}_{\lam} y+\lam \ \left(J_n \ y +b \ u
\ e_n+ d \ e_n\right).
\eeq
 For every $\mu >0$, we also have  that
\[ \frac{\partial D^\rf_{\mu}}{\partial\mu}(D^\rf)^{-1}_{\mu}=\frac{1}{\mu}D_{\rf},
\ \text{ with} \quad D_\rf:=\diag(r_i)_{1\leq i\leq n}. \]

Then \eqref{eq:der0} becomes \beq\label{eq:der1} y'=\left(\frac{
  \dot\lam}{\lam^2}D_{\rf}+J_n\right)y+\left(b(s)\ u(s)+d(s)\right) e_n.
\eeq Here we consider the control $u$ and both $b$ and $f$ as
functions of the new time $s$.

Let $a:[0,T]\to\RR$ be a non negative continuous function so that the
$C^1$ function $A:[0,T]\to\RR$ defined by $A(t)=\int_t^Ta(s)\ ds$ is
positive on $[0,T)$. Setting
\begin{equation}\label{eq:lam0}
\lam(t):=\frac1{A(t)},\quad t\in [0,T),
\end{equation}
one gets that 
$$
\frac{\dot\lam(t)}{\lam^2(t)}=a(t), \quad t\in[0,T).
$$

It is then immediate to see that the function $\lam: [0,T)\to \mathbb{R}_+^*$ is increasing, tends to infinity as $t$ tends to $T$ and the time $s$ defined in \eqref{eq:newT} realizes an increasing $C^1$ bijection from $[0, T)$ to $[0,\infty)$.

We still use $a(s)$ to denote $a(t)$. 
With this choice,  \eqref{eq:der1} becomes
\beq\label{eq:der2}
y'=\left( a(s)D_{\rf}+J_n\right) y+\big(b(s)\ u(s)+d(s)\big) e_n.
\eeq

To solve the original problem of designing a feedback control $u$ that
renders the system FT-ISS-C in time $T >0$, the idea consists in
choosing
\begin{equation}\label{eq:feed0-0}
u=F(y(s)),
\end{equation}
where $F:\mathbb{R}^n\to \mathbb{R}$ is a continuous function to be chosen later.
\begin{remark}
An importance feature in stabilisation of control systems is the fact that one usually requires the feedback law to remain bounded and ideally, to tend to zero as the state $x$ tends to zero, even if in presence of disturbance. In the context of  prescribed time stabilisation of \eqref{eq:PCofI-0}, this feature is automatically guaranteed by our view point of time-varying homogeneity since the feedback law takes the form  \eqref{eq:feed0-0}: bounding $u(t)$ uniformly on $[0,T]$ simply reduces to bound the artificial state $y(s)$ uniformly on $\mathbb{R}_+$.  
\end{remark}
\begin{remark}\label{rem:stabY}
For the stabilisation of \eqref{eq:der2}, one can of course rely on linear feedback laws, as done in the next section (and already done in \cite{Krstic2017}) but also on sliding mode type of feedbacks which insure fixed time (in the scale $s$!) stabilisation with robust properties, see Subsection~\ref{ss:FTF}
 below.
 \end{remark}

\subsection{Linear feedback}\label{ss:linear}

We now revisit the results obtained in \cite{Krstic2017} at the light
of the time-varying homogeneity introduced in the previous section. To establish the connection with that
reference, one must compares our change of variable defined in
\eqref{eq:COC} and the one considered in \cite{Krstic2017}. At once,
one can see that the function $\mu(\cdot)=\frac{T^{m+n}}{(T-t)^{m+n}}$ in Eq. $(31)$ of
\cite{Krstic2017} corresponds, up to a positive constant, to the
time-varying homogeneity parameter $\lambda(\cdot)$ where $a(t)$ is
chosen as $a(t)=(T-t)^{m-1}$ (with $m\geq 2$ integer). In opposite to \cite{Krstic2017}, in our approach one does
not have to take time derivatives of $\lambda(\cdot)$ (or equivalently
of $\mu(\cdot)$), and hence our computations are simpler (in
particular no need of Lemmas $2$ and $3$ in \cite{Krstic2017}).

As for the feedback control in \cite{Krstic2017}, it is given by 
$u = - \frac{1}{b} \left( d + L_0+L_1+ kz \right)$ (cf. Eqs. $(42)$ and $(50)$ in \cite{Krstic2017}),
where $L_0$ is a linear combination of successive derivatives of
$\mu$ and the state components, $L_1$ contains a gain matrix
$K_{n-1}$, $k$ is a scalar gain and $z$ is a change of variable of the $n$-th coordinate of the 
state. The abaove expression of the feedback $u$ shows that this choice of feedback can be
essentially reduced to a linear one (realized by the constant $k$ and
the $\RR^{n-1}$ vector $K_{n-1}$ in \cite{Krstic2017}). This is the reason why we take here $F(y)=-K^Ty$ for some
vector $K\in\RR^n$ to be fixed later. In that case, after replacing $u$ in \eqref{eq:der2}, it follows     
\beq\label{eq:der3} y'=\left(
    a(s)D_{\rf}+J_n-b(s)\ e_nK^T\right) y+d(s)\ e_n, 
    \eeq 
    that is an
    equation of the type $y'=M(s)y+f (s) \ e_n$ where
    $M(s)=a(s)D_{\rf}+J_n-b(s)\ e_nK^T$ with $b(s)$ subject to
    \eqref{eq:b0}. In \cite{CS-2010}, such systems were considered
    (without the term $a(s)D_{\rf}$) and it was proven that there exists a
    positive constant $\mu>0$, a real symmetric positive definite
    $n\times n$ matrix $S>0$ and a vector $K\in\mathbb{R}^n$ such that
    \beq\label{eq:LMI0}
    \Big(J_n-b\ e_nK^T\Big)^TS+S\Big(J_n-b\ e_nK^T\Big)\leq -\mu
    \Id_n,\quad \forall\ \underline{b}\leq b\leq \overline{b}, \eeq
    where $\Id_n$ denotes the $n \times n$ identity matrix and $S$,
    $K$ and $\mu$ depend on $\underline{b}$ and $\overline{b}$.

A careful examination of the argument shows actually that one can remove the upper bound on the parameter $b$. We thus get the slightly stronger result, whose proof is given in Appendix, for sake of completeness.

\begin{proposition}\label{prop:LMI0}
Let $n \in \NN$ and $\underline{b}>0$. Then there exists a positive constant $\rho>0$, a real
symmetric positive definite $n\times n$ matrix $S$ and
$K\in\mathbb{R}^n$ so that \beq\label{eq:LMI1}
\Big(J_n-b\ e_nK^T\Big)^T S + S \Big(J_n-b\ e_nK^T\Big)\leq -\rho
\Id_n,\quad \forall b\geq \underline{b}.  \eeq
\end{proposition}

With an obvious perturbation argument, we immediately derive the following corollary: 
\begin{proposition}\label{prop:Krs1}
Using the notation of Proposition~\ref{prop:LMI0}, there exist $\rho_0, C_0
\in \RR_+^*$, a real symmetric positive definite $n\times n$ matrix
$S$ and $K\in\mathbb{R}^n$ so that \beq\label{eq:LMI2} \Big(a
D_{\rf}+J_n-b\ e_nK^T\Big)^TS+S\Big(a D_{\rf}+J_n-b\ e_nK^T\Big)\leq -\rho_0
Id_n,\quad \forall a\in[-C_0,C_0],\ \ b\geq \underline{b}.
\eeq
\end{proposition}

We apply the above proposition to derive ISS properties of
\eqref{eq:der3} and we prove the following proposition.
\begin{proposition}\label{prop:linear}
Consider the dynamics given in \eqref{eq:der2} and $D^\rf_\eta$ defined in
\eqref{eq:dlam}. Then there exists $K\in \mathbb{R}^n$ and $\eta_1>0$ such that, for
$\eta\geq \eta_1$, the state feedback $u=-K^TD^\rf_\eta y$ provides the
following estimate:
\begin{equation}\label{eq:y0}
\vert y_i(s)\vert\leq \frac
      {\max(1,\eta^{n-1})}{\eta^{n-i}}\exp(-C_S\rho_1\eta\ s)\Vert
      y(0)\Vert+\frac{C_S}{\eta^{n-i+1}}\max_{r\in [0,s]}\vert
      d(r)\vert,\quad \forall s\geq 0, \quad 1\leq i\leq n.
\end{equation}
where $C_S$ and $\rho_1$ are positive constants only depending on the lower
bound $\underline{b}$.
\end{proposition}

\begin{proof}
Fix $\eta>0$ and set $z_\eta=D^\rf_\eta y$. From \eqref{eq:der3}, with
$u=-K_\eta y=-Kz_\eta$, one gets that $z_\eta$ verifies the following
dynamics, after setting the time $\xi:=\eta s$, \beq\label{eq:der4}
\frac{dz_\eta}{d\xi}=\frac{D^\rf_\eta y'}{\eta}=\left(
\frac{a(\xi/\eta)}{\eta}D_{\rf}+J_n- b(\xi/\eta)\ e_nK^T\right) z+
d(\xi/\eta)\ e_n.  \eeq Set $C_a=\sup_{s\geq 0}|a(s)|=\max_{t\in
  [0,T]}|a(t)|$.

One takes the Lyapunov function $V(z_\eta)=z^T_\eta Sz_\eta$
and takes its time derivative along \eqref{eq:der4}. Then, by taking $\eta\geq \frac{C_a}{C_0}$ and using Proposition~\ref{prop:LMI0}, one gets

\begin{align}
\frac{dV(z_\eta(\xi))}{d\xi}&\leq -\rho_0\Vert z_\eta(\xi)\Vert^2+2
\vert z^T_\eta(\xi)Se_n\vert\max_{r\in [0,\xi/\eta]}\vert d(r)\vert\nonumber\\
&\leq -\rho_0\Vert z_\eta(\xi)\Vert^2+C_S
\Vert z^T_\eta(\xi)\Vert\max_{r\in [0,\xi/\eta]}\vert d(r)\vert,\label{eq:derLyap}
\end{align}
where $C_S$ stands for ''any'' constant only depends on $S$, i.e. on $\underline{b}$ and $C_a$.  

One deduces that there exists constants $C_S$ such that 
$$
\Vert z_\eta(\xi)\Vert\leq \exp(-C_S\mu\xi)\Vert z_\eta(0)\Vert+
C_S\max_{r\in [0,\xi/\eta]}\vert d(r)\vert,\quad \forall \xi\geq 0.
$$
We now write the previous inequality in terms of $y(s)$. After noticing that 
$$
\eta^{n-i+1} \vert y_i(s)\vert=\vert (z_\eta)_i(\xi)\vert\leq \Vert z_\eta(\xi)\Vert,\ 1\leq i\leq n,\quad \Vert z_\eta(0)\Vert\leq \eta\max(1,\eta^{n-1})\Vert y(0)\Vert,
$$ 
one gets \eqref{eq:y0}.

\end{proof}
One can rewrite the previous argument using an LMI formulation. For that purpose, one needs a result similar to Proposition~\ref{prop:Krs1}, which involves the extra parameter $\eta$. More precisely, one easily shows the following proposition.
\begin{proposition}\label{prop:Krs2}
Let $n \in \NN$ and $\underline{b} \in \RR_+^*$. Then there exist
a positive constant $\rho$, a real symmetric positive definite
$n\times n$-matrix $S$ and a vector $K\in\mathbb{R}^n$ such that,
for every $C>0$ there exists $\eta_1$ so that, one has
 \beq\label{eq:LMI3}
\Big(a D_{\rf}+J_n-b\ e_nK_\eta^T\Big)^TS_\eta+S_\eta\Big(a D_{\rf}+J_n-b\ e_nK_\eta^T\Big)\leq -\mu_* \eta\ (D^\rf)^2_\eta,\quad  a\in[-C,C],\ \eta\geq\eta_1 \ b\geq \underline{b},
\eeq
where $S_\eta=D^\rf_\eta SD^\rf_\eta$ and
$K_\eta=D^\rf_\eta K$.
\end{proposition}
To see that, simply take $K_\eta=D^\rf_\eta K_1$ and multiply the LMI \eqref{eq:LMI2} on the left and on the right by $D^\rf_\eta$ yields \eqref{eq:LMI3}.

Using Proposition~\ref{prop:linear} and the fact that
$$
\lam(t)^{n-i+1}\vert x_i(t)\leq \vert y_i(s)\vert,\ 1\leq i\leq n,\ t\geq 0,
$$
we deduce PT-ISS-C for $x$ in any time $T>0$.

We gather in the following corollary our findings, which are similar to Theorem $1$ in \cite{Krstic2017}.
\begin{corollary}\label{cor:main-lin}
Consider the dynamics given in \eqref{eq:PCofI-0}. Let $a:[0,T]\to\mathbb{R}$ any  non negativecontinuous function such that $\int_t^Ta(\xi)d\ \xi>0$ for $t\in [0,T)$. Then set
\begin{equation}\label{eq:lam-s-1}
\lam(t)=\frac1{\int_t^Ta(\xi)d\xi}, \quad  s(t)=\int_0^t\lam(\xi)d\xi, \quad 0\leq t<T.
\end{equation}
There exists $K\in \mathbb{R}^n$ such that, for every $\eta\geq 1$, the state feedback 
\begin{equation}\label{eq:feed0}
u=-K^TD^\rf_{\eta\lam(t)}x(t),\quad t\geq 0,
\end{equation}
where $D^\rf_{\lam}$ is defined in\eqref{eq:dlam}, provides the following estimate, for every $t\geq 0$ and $1\leq i\leq n$,
\begin{equation}\label{eq:x0}
\vert x_i(t)\vert\leq\frac1{(\eta\lambda(t))^{n-i+1}}\Big(\eta\max(1,\eta^{n-1})
\exp\left(-C_S\mu\eta\ s(t)\right)\Vert x(0)\Vert+
C_S\max_{r\in [0,t]}\vert d(t)\vert
\Big).
\end{equation}
where $C_S$ and $\mu$ are positive constants only depending on the lower bound $\underline{b}$.
\end{corollary}
\begin{remark}
The case where $a=(T-t)^m$ with $m$ positive integer corresponds to \cite{Krstic2017} and one can choose another $a$ which goes faster to $0$ as $t$ tends to $T$, for instance
$exp(-1/T-t)/(T-t)^2$, which yields to $\lambda(t)=exp(-1/T-t)$ and then faster rates of convergence. 
\end{remark}
One should now refer to Section $3.2$ in \cite{Krstic2017} which provides the advantages and limits of such a feedback regulation. In the sequel, we only insist on what we believe are the advantages of our approach with respect to that of \cite{Krstic2017} as well as the inherent limitations in terms of robustness of feedback strategies based on time-varying homogeneity. 
\begin{remark}\label{rem:compare0}
The disturbance we consider here has a simpler expression than that of \cite{Krstic2017}, the latter being bounded by $\vert d(t)\vert \psi(x)$, with $d$ any measurable function on $[0,\infty)$ and $\psi\geq 0$ a known scalar-valued continuous function. To lighten the presentation, we do not consider the function $\psi$ since the analysis in this case is entirely similar to the above by using Eq. $(25)$ in \cite{Krstic2017}.
\end{remark}
\begin{remark}\label{rem:compare}
Let us compare our results with those obtained in \cite{Krstic2017}. First of all, we recover at once the main result of that reference (Theorem $2$ and Inequality $(79)$)
 by choosing  the function $a(\cdot)$ appearing in the theorem  to be equal to $C(T-t)^{m-1}$ where $C$ is a positive constant and $m$ a positive integer. 
 We have though slightly better results since we can prescribe the rate of exponential decrease as well as the estimate on the error term modeled by $f(\cdot)$ thanks to the occurence of the parameter $\eta$ in our findings. Indeed the choice of the function $\lam(\cdot)$ in \cite{Krstic2017}
 (called $\mu(\cdot)$ in Equation $(30)$ of \cite{Krstic2017})
 must be specific because it relies on the fact that time derivatives of $\mu(\cdot)$ must be expressed as polynomials in $\lam(\cdot)$, cf. Lemmas $2$ to $4$ in the reference therein. Instead, using our presentation, it turns out there is more freedom in the choice of $\lambda$. More importantly, our presentation yields simpler proofs of convergence with a unique time scale for variables and everything boiling down to LMIs.
Another advantage of the more transparent structure of the feedback  is given in the subsequent remarks, where we are able to explain in a very explicit manner
the limitations of the present feedback law, as they are suggested 
in the discussion $3.2$ in \cite{Krstic2017}, as well as in the conclusion of that reference.
\end{remark}
\begin{remark}\label{rem:compare2}
As noticed in \cite{Krstic2017}, the  linear feedback defined in \eqref{eq:feed0} is not suitable if it is subject to measurement noise on $x(\cdot)$. 
More precisely, this amounts to have instead of \eqref{eq:feed0} a feedback 
$\tilde u$ given by 
$$
\tilde u(t)=-K^TD^\rf_{\eta\lam(t)}\big(x(t)+d(t)\big)=u(t)-K^TD^\rf_{\eta\lam(t)}d(t), \quad t\geq 0,
$$ 
i.e., with a disturbance $d$ in 
\eqref{eq:PCofI-0} of the form $\eta\lam(t)\max(1,(\eta\lam)^{n-1}(t))\Vert d(t)\Vert$. We can only derive from Corollary~\ref{cor:main-lin} 
the following estimate, for every
 $t\geq 0$ and $1\leq i\leq n$,
\begin{equation}\label{eq:x2}
\vert x_i(t)\vert\leq \frac{\eta\max(1,\eta^{n-1})}{(\eta\lam)^{n-i+1}(t)}\exp\left(-C_S\mu\eta\ s(t))\right)\Vert x(0)\Vert+C_S \frac{\max(1,(\eta\lam)^{n-1}(t))\max_{r\in [0,t]}\Vert d(r)\Vert}{(\eta\lam)^{n-i}(t)}.
\end{equation}
The right-hand side blows up as $t$ tends to $T$, except for 
$i=1$, with a loss of regulation accuracy (we do not have anymore convergence to zero but to an arbitrary small neighborhood).

On the other hand, by choosing $\eta$ of the amplitude of $\lam(t)$ as $t$ tends to $T$, we deduce at once from \eqref{eq:x2} the following corollary. 
\begin{corollary}\label{cor:minor-lin}
With the notations of Corollary~\ref{cor:main-lin}, assume that 
one feeds the dynamics given in \eqref{eq:PCofI-0} with the pertubed feedback $\tilde u(t)=-K^TD^\rf_{\eta\lam(t)}\big(x(t)+d(t)\big)$.
Then for every time $T'<T$, there exists $\eta>0$ such that, one has that 
\begin{equation}\label{eq:x0}
\max_{t\in [0,T']}\vert x_i(t)\vert\leq \eta \Vert x(0)\Vert+C_{T',T}\max_{t\in [0,T']}\vert d(t)\vert,\quad \forall t\geq 0,\quad 1\leq i\leq n,
\end{equation}
where $C_{T',T}$ is a positive constant, tending to infinity as $T'$ tends to $T$.
\end{corollary}
The previous result of semi-global nature has been already suggested in \cite{Krstic2017} and has been obtained in the present paper thanks to the extra parameter $\eta$. In particular, it follows the idea that in order to obtain estimates for prescribed time control  in time $T'$, one can use the previous strategy of prescribed time control  in a time $T>T'$ and then use \eqref{eq:x0}. This estimate is a direct result of the use of the time-varying function $\lam(\cdot)$ but, as regards measurement noise it is definitely not satisfactory.  
\end{remark}
\begin{remark}\label{rem:compare3}
Looking back at \eqref{eq:feed0-0}, the most natural choice is a linear feedback and it has been (essentially) first addressed in 
\cite{Krstic2017} and revisited here. One can also use other feedback laws, especially those providing finite-time stability (in the scale time $s$). If there is measurement noise, i.e., of the type $x+d$, the feedback implemented in \eqref{eq:der2} will be $u(s)=F(y(s)+D^\rf_{\lambda(t(s))}d(s))$
and it is likely that one can find appropriate perturbations so that the last coordinate of $x$ will become unbounded as $t$ tends to $T$. Already, in the linear case, for a double integrator for instance, it is easy to choose bounded disturbances $d$ in the case where $a(t)=1$ such that $y_2(s)$ has the magnitude of $\lambda^2(s)$, and then $x_2(t)$ has the magnitude of $\lambda(t)$ as $t$ tends to $T$. It is not difficult to extend that remark to any feedback law $F$ which is differentiable at zero. Such a fact prevents to get any type of ISS results and it indicates that no property
such as $(PT-ISS)$ can hold in presence of measurement noise.  This is why time-varying homogeneity based feedbacks are not, in our opinion, well-suited for prescribed-time stabilization in presence of measurement noise. One has to follow another approach and this is the purpose of the next section.
\end{remark}

\subsection{Fixed-time feedbacks}\label{ss:FTF}
We now consider feedback laws in \eqref{eq:feed0-0} which 
will provide fixed-time stabilisation for \eqref{eq:der2}
 under the assumption of a priori knowledge on the uncertainties bounds. More precisely, we will simply show that the feedback law provided in \cite[Theorem 5]{HLCH-2017}
does the job and we have the following.
\begin{proposition}\label{prop:fixed1}
Set $\vep=\pm$. Assume that there exists $\alpha\in (0,1/n)$, $c>0$, two continuous feedback laws $u_\vep:\mathbb{R}^n\to \mathbb{R}$ and two 
$C^1$ functions 
$V_\vep:\mathbb{R}^n\to \mathbb{R}_+$, which are positive definite, $\rf$-homogeneous of degree larger than one and 
such that one has:
\begin{description}
\item[$(a)$] $u_\vep$ stabilizes the $n$-th order pure chain of integrator $\dot x=J_nx+ue_n$ in finite-time and along the trajectories of the corresponding closed-loop system, one has
\begin{equation}\label{eq:vep0}
\dot V_\vep\leq cV_\vep^{1+\vep\alpha};
\end{equation}
\item[$(b)$] for every $x\in \mathbb{R}^n$, the following geometric condition holds true:
\begin{equation}\label{eq:geom}
\frac{\partial V_\vep}{\partial x_n}u_\vep(x)\leq 0
\hbox{ and }u_\vep(x)=0\Rightarrow \frac{\partial V_\vep}{\partial x_n}=0.
\end{equation}
Define the feedback law
$\omega_0:\mathbb{R}^n\to \mathbb{R}$ as
\begin{equation}\label{eq:feed-fix0}
\omega_0(x)=\left\{
 \begin{array}{lll}
 u_+(x),&\hbox{ if }& V_-(x)>1,\\  
 u_-(x),&\hbox{ if } & V_-(x)\leq 1.
\end{array}
\right.
\end{equation}
\end{description}
Consider the dynamics \eqref{eq:der2} and assume that 
$b(\cdot)$ verifies \eqref{eq:b0} and $\Vert d\Vert_{\infty}\leq D$ for some positive constant $D$. Then the feedback law $\omega:\mathbb{R}^n\to \mathbb{R}$ defined as 
\begin{equation}\label{eq:feed-fix}
\omega(x)=\frac1{\underline{b}}\Big(\omega_0(x)+D sgn(\omega_0(x))\Big),
\end{equation}
where $\omega_0$ is defined in \eqref{eq:feed-fix0} and $sgn$ stands for the set-valued function ``sign'', globally stabilises \eqref{eq:der2} in fixed-time.
\end{proposition}  
Here the $sgn$ function makes the closed-loop system corresponding to \eqref{eq:der2} and $u=\omega$ a differential inclusion and its trajectories must be understood in the Filippov sense, cf. \cite{Fil}. Note also that examples of feedbacks
$u_\vep$ and the Lyapunov functions $V_\vep$ verifying Items $(a)$ and $(b)$ are also provided in Definition~\ref{def-Hong-controller} and Proposition~\ref{prop:Hong-gen} given in the next section.
\begin{proof}
First of all notice that, by multiplying \eqref{eq:der2} by $D^\rf_{\mu}$ with $\mu\geq 1$ and considering the new state $D^\rf_{\mu}y(s/\mu)$, $C_a=\sup_{s\geq 0}|a(s)|$ becomes $C_a/\mu$ and hence arbitrary small.

Set $E:=\min_{V_-(x)=1}V_+(x)>0$ and consider the sets 
$$
S_1 =\{x\in\mathbb{R}^n\ :\ V_+(x)\leq E\},\
S_2 =\{x\in\mathbb{R}^n\ :\ V_-(x)\leq 1\}.
$$
By definition of $E$, we have that $S_1\subset S_2$.
We claim that the closed-loop system corresponding to \eqref{eq:der2} and $u=\omega$ is globally fixed-time stable with respect to $S_2$. For that purpose, we compute the time derivative of $V_+$ along the trajectories outside $S_2$ and get
\begin{eqnarray}
\dot V_+&=&a\langle \nabla V_+(y),D_{\rf}y\rangle+
\sum_{i=1}^{n-1}\frac{\partial V_\vep}{\partial x_i}x_{i+1}
+\frac{\partial V_\vep}{\partial x_n}(b\omega+d),\nonumber\\
&\leq& a\langle \nabla V_+(y),D_{\rf}y\rangle-cV_+^{1+\al}
+\frac{\partial V_\vep}{\partial x_n}\Big((\frac{b}{\underline{b}}-1)\omega_0+
sgn(\omega_0)(\frac{b}{\underline{b}}D-\vert d\vert)\Big),
\nonumber\\
&\leq& \frac c2V_+-cV_+^{1+\al}\leq -\frac{c}2V_+^{1+\al}.\label{eq:V+der}
\end{eqnarray}
To get the above we have used Item $(a)$, i.e., 
$$
\sum_{i=1}^{n-1}\frac{\partial V_\vep}{\partial x_i}x_{i+1}+
+\frac{\partial V_\vep}{\partial x_n}\omega_0(x)\leq -cV_+^{1+\al},
$$
Item $(b)$, and the fact that the function $\langle \nabla V_+(y),D_{\rf}y\rangle$ having the same degree of $\rf$-homogeneity as $V$ is smaller than $\frac c2V_+^{1+\al}$ outside $S_2$ for $C_a$ small enough. The claim is proved by using Lemma~\ref{le:notreL1}. 

As soon as a trajectory $x$ of the closed-loop system corresponding to \eqref{eq:der2} and $u=\omega$ reaches $S_2$, it verifies $V_-(x)=1$. Morever for trajectories in $S_2$, a computation entirely similar to \eqref{eq:V+der} yields the differential inequality $\dot V_-\leq -\frac{c}2V_-^{1-\al}$, which proves that any trajectory starting at $V_-(x)=1$ enters in $S_2$, remains in it for all subsequent times and, again according to Lemma~\ref{le:notreL1}, converges to the origin in a uniform finite-time. That concludes the proof of Proposition~\ref{prop:fixed1}.

\end{proof}
 \begin{remark}
 Note that the feedback $\omega$ defined in \eqref{eq:feed-fix}
 exhibits a discontinuity at $V_-=1$. By using the feedback law
 of Theorem~\ref{th:TH1}, one can remove that discontinuity, if in addition, an upper bound for $b$ is assumed to be known. 
 \end{remark}

\section{Robust prescribed-time stabilisation}\label{s:robust}
In the previous section, a linear feedback $u=K^Ty$ was considered but this choice faces a pernicious problem as soon as there is some noise measurement on the state.
We propose in this section an alternative feedback law for prescribed-time stabilisation with ISS properties in presence of measurement noise and unmatched uncertainties. The construction of this feedback runs in two steps, the first one deals with 
the fixed-time stabilisation in the unperturbed case and the second addresses the ISS issue in the perturbed case.
\subsection{A special fixed-time stabilisation design}\label{ss:yet}
The unperturbed case associated with \eqref{eq:PCofI-0}, namely
   \beq\label{eq:n-int-U}
   \dot x=J_nx+u\ e_n,
   \eeq
 which is referred in the sequel as the {\em $n$-th order pure chain of integrators}. 

To proceed, we rely on the original idea of \cite{HLCH-2017} and use the perturbation trick of \cite{LREPP-2018} to provide an explicit and continuous feedback law.

   We next provide the necessary material needed to describe the solution of \cite{HLCH-2017}. The following construction, which is based on a backstepping procedure, has been given first in \cite{Hong-2002} and we will modify it to handle the present situation.
   
   \begin{defn}\label{def-Hong-controller}
 Let $\ell_j>0$, $j=1,\cdots,n$ be positive constants.  For $\kappa \in
\left[-\frac1n,\frac1n\right]$, define the weights
$\mathbf{r}(\kappa)=(r_1,\cdots,r_n)$ by $r_j= 1 + (j-1)\kappa$ , $
j=1,\cdots, n$.  Define the feedback control law
\begin{equation} \label{hong_H} u = \omega_\kappa^{H}(x):= v_n,\end{equation} where the $v_j=v_j(x)$ are defined inductively by:
\begin{equation}\label{eq:v}
 v_0 = 0,\ v_{j} = -\ell_{j} \lceil \lceil x_{j} \rfloor^{\beta_{j-1}
 } - \lceil v_{j-1} \rfloor^{\beta_{j-1} } \rfloor^{\frac{r_j + \kappa}{r_j\beta_{j-1}}},
\end{equation}
and where the $\beta_i$'s are defined by $ \beta_0 =r_2,$ $ (\beta_j +
1)r_{j+1} = \beta_0 + 1 > 0$, $j=1,...,n-1$. 

For $1\leq j\leq n$, 
we also consider the union of the homogeneous unit spheres associated
with $\mathbf{r}$, $\kappa\in\left[-\frac1{2n},\frac1{2n}\right]$,
i.e., 
\begin{equation}\label{eq:Sj}
S^j= \underset{\kappa\in
  \left[-\frac1{2n},\frac1{2n}\right]}{\bigcup} \left\{x\in
\mathbb{R}^j\ \mid\ \vert x_1\vert^{\frac2{r_1}}+\cdots +\vert
x_n\vert^{\frac2{r_n}}=1\right\}.
\end{equation}
Then $S^j$ is clearly a compact
subset of $\mathbb{R}^j$ and dealing with this set constitutes the
main difference with \cite{Hong-2002}.

\end{defn}
We have then the following proposition.
\begin{proposition}\label{prop:Hong-gen}
Then, there exist positive constants $\ell_j>0$, $j=1,\cdots,n$ such that
for every $\kappa\in \left[-\frac1{2n},\frac1{2n}\right]$, the feedback law
$ u=\omega_\kappa^{H}(x)$ defined in \eqref{hong_H} stabilizes the
system \eqref{eq:n-int-U}. Moreover, there exists a homogeneous
$C^1$-function $V_{\kappa}:\mathbb{R}^n\to\R_+$ given by {\small
\begin{align}\label{eq:fdL-Hong}
 V_{\kappa}(x)=\sum_{j=1}^{n} \frac{\left( \left| x_j \right|^{\beta
     _{j - 1} + 1} -\left| v_{j- 1} \right|^{\beta _{j - 1} + 1}
   \right)}{\beta _{j- 1} + 1} - \left\lceil v_{j - 1}
 \right\rfloor^{\beta _{j - 1}} \left( x_j - v_{j- 1} \right),
\end{align}}
which is a Lyapunov function for the closed-loop system
\eqref{eq:n-int-U} with the state feedback $\omega_\kappa^{H}$, and it satisfies
\begin{equation}\label{eq:diffV}
\dot V_{\kappa} \le -C V_{\kappa}^{1+\alpha(\kappa)},
\qquad \alpha(\kappa):=\frac{\kappa}{2 + \kappa},
\end{equation}
 for some positive constant $C$, independent of $\kappa$. Moreover, $V_{\kappa}$ is $\mathbf{r}(\kappa)$-homogeneous of degree $(2+\kappa)$ with respect to the family of dilations $\big(D^{\mathbf{r}(\kappa)}_{\lambda}\big)_{\lambda>0}$.
\end{proposition}

\begin{remark}\label{rem:V0}
\begin{description}
\item[$(i)$] The previous proposition is essentially Theorem $3.1$ of \cite{Hong-2002}, except that the gains $\ell_i$ are uniform with respect to 
$ \kappa\in\left[-\frac1{2n},\frac1{2n}\right]$. The choice of $\frac1{2n}$ has been made because the previous proposition actually holds true for $ \kappa\in\left[-\frac1{n},\frac1{n}\right]$ at the exception that $V_{\frac1n}$ is not $C^1$ on $\mathbb{R}^n$.
\item[$(ii)$] The critical exponent $1+\alpha(\kappa)$ appearing in \eqref{eq:diffV} is larger than one if $\kappa>0$ and smaller than one if $\kappa<0$.  
\item[$(iii)$] Note also that for $\kappa=0$, one gets a linear feedback and $V_0$ is a positive definite quadratic form, hence there exists  a real symmetric positive definite $n\times n$ matrix $P$ such that $V_0(x)=x^TPx$ for every $x\in\mathbb{R}^n$. Finally,  the time derivative of $V_0$ is associated with
 the $n\times n$ matrix $L^TP+PL$ where $L$ is the companion matrix associated with the coefficients $l_1,\cdots,l_n$. We deduce at once that $L$ is Hurwitz since the differential inequality \eqref{eq:diffV} for $\kappa=0$ is equivalent to the LMI, $\ A^TP+PA\leq -CP$. We set $Q:=-(A^TP+PA)$, which is a real symmetric positive definite $n\times n$ matrix.
\end{description}
\end{remark}
\begin{proof} The argument follows closely that of Theorem $3.1$ of \cite{Hong-2002}, but we will bring some technical  changes to obtain the required uniformity with respect to $\kappa\in  \left[-\frac1{2n},\frac1{2n}\right]$. Moreover, in order to show in the next section the explicit character of our construction, we will provide quantitative estimates on the several constants involved in the construction, which are new with respect to \cite{Hong-2002}.

Let $\kappa\in\left[-\frac1{2n},\frac1{2n}\right]$ and set, for $1\leq j\leq n$, $x^{(j)}=(x_1,\cdots,x_j)$, 
\begin{eqnarray}
W_{\kappa,j}(x^{(j)})&=&\int_{v_{j - 1}}^{x_j}
\big(\left\lfloor s \right\rceil^{\beta _{j - 1}}-\left\lfloor v_{j - 1} \right\rceil^{\beta _{j - 1}}\big)ds \nonumber\\
&=&\frac{1}{\beta _{i - 1} + 1} \left( \left| x_j \right|^{\beta _{j - 1} + 1} -\left| v_{j - 1} \right|^{\beta _{j- 1} + 1}  \right) - \left\lfloor v_{j - 1} \right\rceil^{\beta _{j - 1}} \left( x_j - v_{j- 1} \right),\label{eq:Wj}
\end{eqnarray}
and
\begin{equation}\label{eq:Vj}
V_{\kappa,0}:=0,\quad
\hbox{ and }V_{\kappa,j}:=W_{\kappa,j}+V_{\kappa,j-1}.
\end{equation}
One has $V_{\kappa}(x)=V_{\kappa,n}(x^{(n)})=\sum_{i=1}^{n}W_{\kappa,j}(x^{(j)})$. The choice of the $\ell_j$ is made recursively at each step $1\leq j\leq n$
by considering, as in \cite{Hong-2002}, the following expression
\begin{equation}\label{eq:derVj}
\frac{d}{dt}V_{\kappa,j}=\frac{d}{dt}V_{\kappa,j-1}+
\frac{\partial V_{\kappa,j-1}}{\partial x_{j-1}}(x_{j}-v_{j-1})+
\sum_{i=1}^{j-1}\frac{\partial W_{\kappa,j}}{\partial x_i}x_{i+1}+
\frac{\partial W_{\kappa,j}}{\partial x_j}v_j,
\end{equation}
where $\frac{d}{dt}V_{\kappa,j-1}$ is used to denote the time derivative of $V_{\kappa,j-1}$ is taken along the $(j-1)$th pure chain of integrators and the functions $\frac{\partial V_{\kappa,j-1}}{\partial x_{j-1}}$ and $\frac{\partial W_{\kappa,j}}{\partial x_i}$ are continuous.

We also get that, for $1\leq j\leq n$, one has 
\begin{eqnarray}
\frac{\partial V_{\kappa,j-1}}{\partial x_{j-1}}(x_{j}-v_{j-1})&=&
\frac{\partial W_{\kappa,j-1}}{\partial x_{j-1}}(x_{j}-v_{j-1})=
\big(\left\lfloor x_{j-1} \right\rceil^{\beta _{j - 2}}-\left\lfloor v_{j - 2} \right\rceil^{\beta _{j - 2}}\big)(x_{j}-v_{j-1})\label{eq:Vj1}\\
\frac{\partial W_{\kappa,j}}{\partial x_j}v_j&=&- \ell_j Z_{\kappa,j},\ \ 
Z_{\kappa,j}=\big\vert \left\lfloor x_j\right\rceil^{\beta_{j-1}}-  \left\lfloor v_{j-1}\right\rceil^{\beta_{j-1}}\big\vert^{2(1+\kappa)/r_j\beta_{j-1}}.\label{eq:key}
\end{eqnarray}

We will need the following elementary fact: 
for every $\alpha$ in a compact set of $\mathbb{R}_+^*$ and $M>0$, there exists two positive constants $A,B$ such that, for every real numbers $\vert x\vert, \vert y\vert\leq M$, one has
\begin{equation}\label{eq:elem3}
A\vert x-y\vert^{\max(1,\alpha)}\leq  \vert \left\lfloor x\right\rceil^{\alpha}-\left\lfloor y\right\rceil^{\alpha}\vert\leq
B\vert x-y\vert^{\min(1,\alpha)}.
\end{equation}

We next prove by induction on $1\leq j\leq n$, that there exists 
positive real numbers $\ell_1,\cdots, \ell_n$ 
such that 
\begin{equation}\label{eq:derVj}
\max\{\frac{d}{dt}V_{\kappa,j}\ \mid\ 
-\frac1{2n}\leq\kappa\leq\frac1{2n},\ x^{(j)}\in S^j\}
\leq -\frac{l_1}{2^{j-1}}.
\end{equation}
By homogeneity and for $j=n$, one immediately gets \eqref{eq:diffV} and the conclusion of the proposition. 

In the rest of the argument, we use $K_j,M_j,L_j$ to denote positive constants depending on $S^j$ and $\ell_1,\cdots,\ell_{j-1}$ but independent of $\ell_j$. For $j=1$, \eqref{eq:derVj} reduces to $\frac{d}{dt}V_{\kappa,1}=-l_1$ and any positive $l_1$ does the job.
For the inductive step with $2\leq j\leq n$, assume that $\ell_1,\cdots,\ell_{j-1}$ have been built with the required properties, in particular we have
$\frac{d}{dt}V_{\kappa,j-1}\leq -\frac{l_1}{2^{j-2}}$ on $S^{j-1}$.

From \eqref{eq:Vj1}, we get
\begin{equation}\label{eq:Vk1}
\big\vert \frac{\partial V_{\kappa,j-1}}{\partial x_{j-1}}(x_{j}-v_{j-1})
\big\vert\leq K_j\vert x_{j}-v_{j-1}\vert.
\end{equation}
For $1\leq i\leq j-1$, the continuous function $\frac{\partial W_{\kappa,j}}{\partial x_i}$ is $\mathbf{r}(\kappa)$-homogeneous of degree $(2+\kappa)$ with respect to the family of dilations $\big(D^{\mathbf{r}(\kappa)}_{\varepsilon}\big)_{\varepsilon>0}$. (Actually, one restricts this homogeneity to $x^{(j)}$.) Moreover, it is equal to zero if $x_{j}=v_{j-1}$. Hence, by using repeatedly \eqref{eq:elem3} and noticing that $v_1,\cdots,v_{j-1}$ do not depend on $\ell_j$, one deduces that there exists $L_j,M_j>0$ such that, for every $x^{(j)}\in S^j$, if $\tilde{\beta}_j=\min(1,\beta_{j-1})$, one has, for every $\kappa\in[-\frac1{2n},\frac1{2n}]$ and $x^{(j)}\in S^j$,
\begin{eqnarray}
\big\vert\sum_{i=1}^{j-1}\frac{\partial W_{\kappa,j}}{\partial x_i}x_{i+1}\big\vert &\leq& L_j\vert x_{j}-v_{j-1}\vert^{\tilde{\beta}_j},\label{eq:Wj1}\\
\vert Z_{\kappa,j}\vert&\geq& M_j\vert x_{j}-v_{j-1}\vert^{2(1+\kappa)/r_j\tilde{\beta_j}}.   \label{eq:Zj1}
\end{eqnarray}
(Note that we used in the above that $\vert x_j\vert\leq 1$ for $1\leq j\leq n$ and $x^{(j)}\in S^j$ as well as a bound on the $v_j$ obtained with an immediate inductive argument based on \eqref{eq:v}.)

Inserting \eqref{eq:Vk1}, \eqref{eq:Wj1} and \eqref{eq:Zj1} in \eqref{eq:derVj}, one deduces that,  
\begin{equation}\label{eq:derVj1}
\frac{d}{dt}V_{\kappa,j}\leq -\frac{l_1}{2^{j-2}}+(K_j+L_j)\vert x_{j}-v_{j-1}\vert^{\tilde{\beta}_j}
-l_jM_j\vert x_{j}-v_{j-1}\vert^{2(1+\kappa)/r_j\tilde{\beta_j}}.
\end{equation}
Set $\xi_j=\big(\frac{l_1}{(K_j+L_j)2^{j-1}})^{1/\tilde{\beta}_j}$. By definition, one gets that
$\frac{d}{dt}V_{\kappa,j}\leq -\frac{l_1}{2^{j-1}}$ if $\vert x_{j}-v_{j-1}\vert\leq \xi_j$.
Now, if $\vert x_{j}-v_{j-1}\vert>\xi_j$, one chooses $l_j$ such that
\begin{equation}\label{eq:ell}
\ell_j\geq \frac{(K_j+L_j)}{M_j\xi_j^{2(1+\kappa)/r_j\tilde{\beta_j}-1/\tilde{\beta}_j}}.
\end{equation}
This is possible since the right-hand side of the above inequality does not depend on $\ell_j$. In that case, $\frac{d}{dt}V_{\kappa,j}\leq -\frac{l_1}{2^{j-2}}$. This concludes the proof of the inductive step.

\end{proof}
\begin{remark}\label{rem:exp0}
One can notice in the above argument a difference with respect of that of \cite{Hong-2002} which consists in introducing the constants $K_j,L_j$ and $M_j$. The latter provide an explicit choice  in order to be as explicit as possible in view of numerical determination of the constants $\ell_1,\cdots,\ell_n$.
\end{remark}

We next consider a state varying homogeneity degree given next.
 \begin{defn}\label{def-Our--controller}
 For $m\in (0,1)$ and $\kappa_0\in (0,\frac1{2n})$, define the following continuous function $\kappa:\mathbb{R}^n\to [-\kappa_0,\kappa_0]$ by
 \begin{equation}\label{eq:feed00}
 \kappa(x)=\left\{
 \begin{array}{lll}
 \kappa_0,&\hbox{ if }& V_0(x)>1+m,\\  
 \kappa_0\big(1+\frac{V_0(x)-(1+m)}{m}\big),&\hbox{ if } &1-m\leq V_0(x)\leq 1+m,\\
 -\kappa_0,&\hbox{ if } &V_0(x)<1-m.
\end{array}
\right.
\end{equation}
 \end{defn}
 We also need the following notation. For $\kappa\in [-\frac1{2n},\frac1{2n}]$ and $a,b$ non negative real numbers, let $B^{\kappa}_{a,b}$, $B^{\kappa}_{\leq a}$ 
and $B^{\kappa}_{\geq b}$ respectively be the subsets of $\mathbb{R}^n$ defined respecvely by
\begin{eqnarray*}
B^{\kappa}_{a,b}&:=&\{x\in\mathbb{R}^n,\ \mid\ a\leq V_{\kappa}(x)\leq b\},\\
B^{\kappa}_{< a}&:=&\{x\in\mathbb{R}^n,\ \mid\ V_{\kappa}(x)<a\},\\
B^{\kappa}_{> b}&:=&\{x\in\mathbb{R}^n,\ \mid\ b< V_{\kappa}(x)\},\\
B^{\kappa}_a&:=&\{x\in\mathbb{R}^n,\ \mid V_{\kappa}(x)=a\}.
\end{eqnarray*}
The last set corresponds to the weighted spheres associated with the positive definite functions $V_\kappa$.

In the spirit of \cite{LREPP-2018}, we are now able to define the introduce the feedbacks which will ultimately yield prescribed time stability. We have the following result.

\begin{Theorem}\label{th:TH1}
Assume that the uncertainty $b$ is bounded, i.e., one has
\begin{equation}\label{eq:bb}
\overline{b}\geq b(t)\geq \underline{b},\quad t\geq,
\end{equation}
for some positive constants $\overline{b},\underline{b}$. Then, 
there exists $m\in (0,1)$ and $\kappa_0\in (0,\frac1{2n})$ such that, the undisturbed $n$-th order chain of integrators defined by
\begin{equation}\label{eq:n-int-ud}
\dot x(t)=J_nx(t)+b(t)u(t),\quad \overline{b}\geq b(t)\geq \underline{b},\quad t\geq 0,
\end{equation} 
together with  an adapted feedback law given by $\omega_{\kappa(x)}^{H}(x)$, with $\kappa(\cdot)$  defined in \eqref{eq:feed00} is globally fixed-time stable at the origin in at most time $T(m,\kappa_0)$ upper bounded as
\begin{equation}\label{eq:est-globalF}
T(m,\kappa_0)\leq
\frac1{C}\Big(\frac{r(m,\kappa_0)^{-\alpha(\kappa_0)}}{\alpha(\kappa_0)}-2\ln(2m)+\frac{r(m,-\kappa_0)^{-\alpha(-\kappa_0)}}{-\alpha(-\kappa_0)}\Big),
\end{equation}
where $r(m,\kappa_0)>0$ (and $r(m,-\kappa_0)>0$) is the largest (smallest) number $r>0$ such that $B^{\kappa_0}_{< r}$ ($B^{-\kappa_0}_{< r}$) is contained in (contains) $B^{0}_{<1+m}$ ($B^{0}_{< 1-m}$) and the constant $C$ has been introduced in \eqref{eq:diffV}.
\end{Theorem}
By adapted, we mean the following: strictly speaking, we must choose the feedback law $\omega_{\kappa(x)}^{H}(x)/\underline{b}$. However, we can replace 
$\ell_n$ by either $\ell_n/\underline{b}$ or by $\underline{b}\ell_n$
 in order to satisfy \eqref{eq:ell}. Hence, with no loss of generality, we assume $\underline{b}=1$.
\begin{proof}
For this result, we follow the perturbative argument considered in the proof of  Lemma 2 in \cite{LREPP-2018}. For that purpose, the time derivative of $V_0$ along non trivial trajectories of System \eqref{eq:n-int-ud} closed by  the feedback law given by $\omega_{\kappa(x)}^{H}(x)$ can be written as
\begin{equation}\label{eq:diff0}
\dot V_0=2x^TP(J_nx+b\omega_{\kappa(x)}^{H}(x)e_n)\leq -x^TQx+2\vert x^TPe_n\vert \delta(x), \ \delta(x):=\overline{b}
\vert \omega_{\kappa(x)}^{H}(x)-\omega_0^{H}(x)\vert.
\end{equation}
We have to first to show that trajectories of 
\begin{equation}\label{eq:closedS0}
 \dot x=J_nx+\omega_{\kappa(x)}^{H}(x)\ e_n, 
 \end{equation}
are well-defined and second that trajectories starting in $B^0_{> 1+m}$
reach $B^0_{1+m}$ in finite time, then ''cross'' it till reaching $B^0_{1-m}$ in finite time and finally remain in  $B^0_{< 1-m}$ for all larger times, while converging to zero in finite time.

Since the right-hand side of \eqref{eq:closedS0} is continuous, there
exist solutions from any initial condition defined at least on a non
trivial interval. Clearly, there exists $R>0$ such that trajectories
starting at any $x_0\in B^{\kappa_0}_{>R}$ stay in the compact set
$B^{\kappa_0}_{<V_{\kappa_0}(x_0)}$ and hence are defined for all
times.

Both the convergence parts of the claim follow from the arguments of 
\cite{BB-2005} and Lemma~\eqref{le:notreL1}, where one proves the following
\begin{itemize}
\item the closed-loop system \eqref{eq:closedS0} is $\mathbf{r}(\kappa_0)$-homogeneous of degree $2+\kappa_0$ in $B^0_{> 1+m}$ and hence converges in finite-time to $B^0_{1+m}$,
 \item the closed-loop system \eqref{eq:closedS0} is $\mathbf{r}(-\kappa_0)$-homogeneous of degree $2-\kappa_0$ in $B^0_{< 1-m}$ and hence converges in finite-time to the origin. 
 \end{itemize}
For the remaining part of the argument, it amounts to show that, 
for $m\in (0,1)$ and $\kappa_0\in (0,1/n)$ small enough, the time derivative of $V_0$
along trajectories of \eqref{eq:closedS0} is negative in $B^0_{1-m,1+m}$.
To see that, it is enough to notice that the function $\delta$ defined in \eqref{eq:diff0} is continuous and 
 tends to zero if either $m$ or $\kappa_0$ tends to zero. 

It remains to provide a first quantitative estimate of the ''fixed-time'' part of the theorem.
 One has that the time needed for the closed-loop system \eqref{eq:closedS0} to converge to $B^0_{1+m}$ is at most equal to the time $T_+(m,\kappa_0)$ needed to converge to  $B^{\kappa_0}_{< r(m,\kappa_0)}$. By integrating \eqref{eq:diffV}, one derives that 
 $$
 T_+(m,\kappa_0)\leq \frac1{C\alpha(\kappa_0)r(m,\kappa_0)^{\alpha(\kappa_0)}}.
 $$
 A similar reasoning yields that the time $T_-(m,\kappa_0)$ needed to converge  from $B^{-\kappa_0}_{< r(m,-\kappa_0)}$ to the origin verifies the following
 $$
 T_-(m,\kappa_0)\leq \frac1{-C\alpha(-\kappa_0)r(m,-\kappa_0)^{\alpha(-\kappa_0)}}.
 $$
(Recall that $\alpha(-\kappa_0)=\frac{-\kappa_0}{2-\kappa_0}<0$.)
It remains to upper bound the time $T_0(m,\kappa_0)$ needed to ``cross'' $B^0_{1-m,1+m}$. For that purpose, choose $m$ and $\kappa_0$ small enough so that 
\begin{equation}\label{eq:est-cross}
M_\delta(m,\kappa_0):=\max_{x\in B^0_{1-m,1+m}}\vert \delta(x)\vert\leq \frac{C(1-m)}2.
\end{equation}
In that case, \eqref{eq:diff0} becomes 
$\dot V\leq -CV/2$ and one gets
$$
T_0(m,\kappa_0)\leq \frac{-2\ln(2m)}{C}.
$$
We conclude that  the closed-loop system \eqref{eq:closedS0} is globally-fixed time stable with respect to the origin in settling time less than or equal to $T(m,\kappa_0)$ given by 
$$
T(m,\kappa_0):=T_+(m,\kappa_0)+T_0(m,\kappa_0)+T_-(m,\kappa_0).
$$
One has then \eqref{eq:est-globalF} and this concludes the proof of the theorem.

\end{proof}

\begin{remark}The above result is the counterpart of Lemma 2 in \cite{LREPP-2018} for our feedback law  $\omega_\kappa^{H}$. Note that in that reference, the statements of Lemma 2 and Theorem 4 as well as the argument of Lemma 2 consider the euclidean norm $\Vert x\Vert$ instead of $B^0_{1}$ in the definition of $\kappa(\cdot)$. As one can see from the above argument, using that norm cannot not provide the required results. However \cite{LREPP-2018} does consider the correct controller in Lemma 3 and in the last section of the corresponding reference.
\end{remark}
It remains to use a standard time rescaling technic with homogeneity (cf. \fixmeyc{\cite{Hong-and-al-2005}} and \cite{HLCH-2017}) to pass from the result of fixed-time stability contained in Theorem~\ref{th:TH1} to a result about prescribed-time stability.
\begin{Theorem}\label{th:TH2}
Let $m\in (0,1)$, $\kappa_0\in (0,1/n)$ defined in Theorem~\ref{th:TH1}
and the feedback law $\omega_{\kappa(x)}^{H}(x)$ defined in \eqref{eq:feed00}
which renders System \eqref{eq:n-int-U} globally fixed-time stable at the origin in settling time less than or equal to $T(m,\kappa_0)$ defined in \eqref{eq:est-globalF}. Then, given any $T>0$, the 
the feedback law $\omega_{\kappa(D_\lambda x)}^{H}(D^\rf_{\lambda} x)$ renders System \eqref{eq:n-int-U} globally fixed-time stable at the origin in settling time less than or equal to $T$ as soon as $\mu\geq T(m,\kappa_0)/T$.
\end{Theorem}
\begin{proof}
For $\lambda>0$, one sets $y(s)=D_\lambda x(t)$ with the time scale $s=\lambda t$. One deduces at once that $y$ converges in finite time to the origin with a settling time upper bounded by $T(m,\kappa_0)$ as well as  $x$, with a settling time upper bounded by $T(m,\kappa_0)/\lambda$. To guarantee that the latter is less than or equal to $T$, it is enough to choose $\lambda$ as stated. 

\end{proof}

\subsection{Explicit determination of the main parameters}\label{ss:explicit}
In order to fully compare our controller $u=\omega^H_{\kappa(x)}$, with $x\mapsto \kappa(x)$ given in Definition~\ref{def-Our--controller} with the controller provided in  \cite{LREPP-2018}, we must explain how to choose the parameters $m\in (0,1)$ and $\kappa_0\in \left(0,\frac1{2n}\right)$ introduced in Definition~\ref{def-Our--controller}. 
We also have to estimate the quantities $r(m,\kappa_0)$ and $r(m,-\kappa_0)$ in order to get a hold on the upper bound $T(m,\kappa_0)$ of the settling time to reach precise estimates of the rescaling factor $\lambda$ appearing in Theorem~\ref{th:TH2}. 

For that purpose, we first need an explicit bound on the coordinates of $x\in B^\kappa_{1-m,1+m}$ with $-\frac1{2n}\leq\kappa\leq\frac1{2n}$. This is the content of the next lemma.
\begin{lemma}\label{lem:est-x}
Let $m\in (0,1)$. Then, there exists an explicit positive constant $X_n$ (depending on $m$ and the $\ell_j$'s)
such that, for $0\leq j\leq n$, $x\in B^\kappa_{1-m,1+m}$ and $-\frac1{2n}\leq\kappa\leq\frac1{2n}$, $\vert x_j\vert,\vert v_{j}\vert \leq X_n$.
\end{lemma}
\begin{proof}
Fix $m\in (0,1)$, $x\in B^\kappa_{1-m,1+m}$ and $-\frac1{2n}\leq\kappa\leq\frac1{2n}$.
The proof of the lemma goes by induction on $j$, where we prove the statement with a constant $X_j$ explicitly depending on $m$ and the $\ell_j$'s.

This is clearly true for $j=1$ since 
$\vert v_1\vert=\ell_1\vert x_1\vert^{1+\kappa}$ and $\frac{\vert x_1\vert^{1+\beta_0}}{1+\beta_0}\leq V_\kappa(x)\leq 1+m$. Assume that the thesis holds true for $j-1\geq 1$. 
One then deduces from the definition of $W_{\kappa,j}$ in \eqref{eq:Wj} and the induction hypothesis that
$$
\vert x_j\vert^{1+\beta_{j-1}}\leq (2+\beta_{j-1})X_{j-1}^{1+\beta_{j-1}}+(1+\beta_{j-1})X_{j-1}^{\beta_{j-1}}\vert x_j\vert+(1+\beta_{j-1})(1+m).
$$
Since $\beta_{j-1}>0$, one deduces at once a first explicit bound for $x_j$ and then for $v_j$ by using \eqref{eq:v}. 

\end{proof}

The following lemma provides the required differences between useful quantities evaluated at any $\kappa\in [-\frac1{2n},\frac1{2n}]$ and $\kappa=0$.
For $0\leq j\leq n$, we introduce the notation $v^{\kappa}_j:=v_j$, where the latter has been defined in \eqref{eq:v}. 
\begin{lemma}\label{lem:est-om}
Let $m\in (0,1)$. Then there exists explicit positive constants $C^1_n,C^2_n$ (depending on $m$ and the $\ell_j$'s) such that, 
\begin{equation}\label{eq:diff-om}
\max\{\vert \omega_{\kappa}^{H}(x)-\omega_0^{H}(x)\ \mid\ -\frac1{2n}\leq\kappa\leq\frac1{2n},\ x\in B^\kappa_{1-m,1+m}\}
\leq C^1_n\vert \kappa\vert^{\min(1,r_n)},
\end{equation}
and 
\begin{equation}\label{eq:diff-V}
\max\{\vert V_{\kappa}(x)-V_0(x)\ \mid\ -\frac1{2n}\leq\kappa\leq\frac1{2n},\ x\in B^\kappa_{1-m,1+m}\}
\leq C^2_n\vert \kappa\vert^{\min(1,r_n)}.
\end{equation}
\end{lemma}
\begin{proof}
Fix $m\in (0,1)$. We will actually prove by induction on $1\leq j\leq n$, that there exists an explicit positive constant $C^1_j$ (depending on $m$ and $\ell_1,\cdots,\ell_j$) such that, 
\begin{equation}\label{eq:diff-om}
\max\{\vert v_j^{\kappa}(x)-v_j^{0}(x)\ \mid\
-\frac1{2n}\leq\kappa\leq\frac1{2n},\ x^{(j)}\in B^\kappa_{j,m}\}
\leq C^1_j\vert \kappa\vert^{\min(1,r_j)},
\end{equation}
where $B^\kappa_{j,m}$ is the set of $x^{(j)}\in\mathbb{R}^j$ for which $1-m\leq V_{\kappa,j}(x^{(j)})\leq 1+m$.

The result is immediate for $j=0$ and hence we turn to the inductive step for $1\leq j\leq n$, assuming that the hypothesis holds for $j-1$. 

Let $-\frac1{2n}\leq\kappa\leq\frac1{2n}$ and $x^{(j)}\in B^\kappa_{j,m}$. Then
$$
v_j^{\kappa}(x)-v_j^{0}(x)=-\ell_j
\lceil \lceil x_{j} \rfloor^{\beta_{j-1}
 } - \lceil v^{\kappa}_{j-1} \rfloor^{\beta_{j-1} } \rfloor^{\frac{r_{j+1}}{r_j\beta_{j-1}}}+\ell_j(x_j-v^{0}_{j-1} )=-\ell_j(F+G),
$$
where 
\begin{eqnarray*}
F&=&
\lceil \lceil x_{j} \rfloor^{\beta_{j-1}
 } - \lceil v^{\kappa}_{j-1} \rfloor^{\beta_{j-1} } \rfloor^{\frac{r_{j+1}a}{r_j\beta_{j-1}}}-
\lceil \lceil x_{j} \rfloor^{\beta_{j-1}
 } - \lceil v^{0}_{j-1} \rfloor^{\beta_{j-1} } \rfloor^{\frac{r_{j+1}a}{r_j\beta_{j-1}}},\\
 G&=&
\lceil \lceil x_{j} \rfloor^{\beta_{j-1}
 } - \lceil v^{0}_{j-1} \rfloor^{\beta_{j-1} } \rfloor^{\frac{r_{j+1}}{r_j\beta_{j-1}}}-(x_j-v^{0}_{j-1} ).
 \end{eqnarray*}
By applying \eqref{eq:elem3} with $\alpha=\frac{r_{j+1}}{r_j\beta_{j-1}}$, then with $\alpha=\beta_{j-1}$ and $A,B$ depending on $X_n$ obtained in Lemma~\ref{lem:est-x}, we get 
\begin{equation}\label{eq:F}
\vert F\vert\leq B\Big\vert \lceil v^{\kappa}_{j-1} \rfloor^{\beta_{j-1}}-\lceil v^{0}_{j-1} \rfloor^{\beta_{j-1}}\Big\vert^{\min(1,\frac{r_j + \kappa}{r_j\beta_{j-1}})}\leq 
B^2\vert v^{\kappa}_{j-1}- v^{0}_{j-1}\vert^{\nu_j},
\end{equation}
where $\nu_j:=\min(1,\beta_{j-1})\min(1,\frac{r_j + \kappa}{r_j\beta_{j-1}})\leq 1$.

To bound $G$, we consider 
\begin{eqnarray*}
M&:=&\max(\vert x_j\vert, \vert v^{0}_{j-1}\vert),\
\varepsilon:=sign(x_jv^{0}_{j-1})\in\{-1,1\},\\
\tau&:=&\frac{\max(\vert x_j\vert, \vert v^{0}_{j-1}\vert)}M\in [0,1],\
N:=1+\varepsilon t^{\beta{j-1}},
\end{eqnarray*}
where we have assumed with no loss of generality that $M>0$.

An easy computation yields that
\begin{equation}\label{eq:G0}
G=(M^{\frac{r_{j+1}}{r_j}}-M)N^{\frac{r_{j+1}}{r_j\beta_{j-1}}}
+M\big(N^{\frac{r_{j+1}}{r_j\beta_{j-1}}}-N+\varepsilon (t^{\beta_{j-1}}-t)\big).
\end{equation}

We now use the following elementary fact: for non negative $x$ and $\alpha>0$,
one has 
$$
\vert x^{\alpha}-x\vert\leq \vert \alpha-1\vert \ln(x)x^{\min(1,\alpha)}.
$$
By applying that fact to \eqref{eq:G0}, we deduce that there exists an explicit positive constant $D_j$ (depending on $m$ and $\ell_1,\cdots,\ell_{j-1}$) such that 
$\vert G\vert\leq \vert D_j\vert \kappa\vert$. From \eqref{eq:F} and the previous inequality, we get that
$$
\vert v_j^{\kappa}(x)-v_j^{0}(x)\vert\leq \ell_l\big(B^2\vert v^{\kappa}_{j-1}- v^{0}_{j-1}\vert^{\nu_j}+D_j\vert \kappa\vert\big).
$$
By applying the induction hypothesis on $\vert v^{\kappa}_{j-1}- v^{0}_{j-1}\vert$,
we prove the inductive step with $C^1_j:=l_j(B^2C_{j-1}^{\nu_{j-1}}+D_j)$. This concludes the proof of \eqref{eq:diff-om}.

We now turn to the proof of \eqref{eq:diff-V}. It is enough to prove the result for one single $W_{\kappa,j}$. Hence let $-\frac1{2n}\leq\kappa\leq \frac1{2n}$ and $x^{(j)}\in B^\kappa_{j,m}$. One has 
\begin{eqnarray*}
W_{\kappa,j}(x^{(j)})-W_{0,j}(x^{(j)})&=&
\frac{\beta_{j-1}-1}{\beta_{j-1}+1}\Big(\vert v_{j-1}^{\kappa}\vert^{\beta_{j-1}+1}-\vert v_{j-1}^{0}\vert^{\beta_{j-1}+1}\Big)\\
&+&\frac12\Big(\vert x_{j}^{\kappa}\vert^{\beta_{j-1}+1}-x_j^2-(\vert v_{j-1}^{\kappa}\vert^{\beta_{j-1}+1}-(v_{j-1}^0)^2)\Big)\\
&-&x_j\Big( \lceil v^{\kappa}_{j-1} \rfloor^{\beta_{j-1}}- \lceil v^{0}_{j-1} \rfloor^{\beta_{j-1}}+(v^{\kappa}_{j-1}-v^{0}_{j-1})\Big)\\
&+&\vert v_{j-1}^{\kappa}\vert^{\beta_{j-1}+1}-(v_{j-1}^0)^2).
\end{eqnarray*}
Following the same type of estimates used to derive \eqref{eq:diff-om}, one gets \eqref{eq:diff-V}.

\end{proof}

We can now provide explicit bounds on $\kappa_0$, for the results of the previous section to hold.
\begin{proposition}\label{prop:estim-k}
Let $m\in (0,1)$. Then there is an explicit $\kappa_0(m)\in [-\frac1{2n},-\frac1{2n}]$ such that, for every $\kappa_0\in (0,\kappa_0(m))$, the statements of Theorem~\ref{th:TH1} and Theorem~\ref{th:TH2}  hold true.
\end{proposition}
\begin{proof}
To determine $\kappa_0(m)$, we rewrite \eqref{eq:diff0} as follows,
$$
\dot V_0\leq -CV_0+2\sqrt{V_0}\sqrt{V_0(e_n)}\vert \delta(x)\vert.
$$
The constant $C$ above has been characterized in \eqref{eq:diffV}.

Along trajectories of System \eqref{eq:n-int-U} closed by  the feedback law given by $\omega_{\kappa(x)}^{H}(x)$ inside $B^0_{1-m,1+m}$, one gets by using $1-m\leq V_{0}(x)\leq 1+m$ and \eqref{eq:diff-om} that 
$$
CV_0\geq C(1-m),\quad 2\sqrt{V_0}\sqrt{V_0(e_n)}\vert \delta(x)\vert\leq 2\sqrt{1+m}\sqrt{V_0(e_n)}C^1_n
\kappa_0^{1-(n-1)/2n}.
$$
One chooses then $\kappa_0(m)>0$ so that $\dot V_0\leq -\frac{CV_0}2$ inside $B^0_{1-m,1+m}$, which yields that 
$$
\kappa_0(m):=\Big(\frac{C(1-m)}{4\sqrt{1+m}\sqrt{V_0(e_n)}C^1_n}\Big)^{\frac{2n}{n+1}}.
$$
As for Theorem~\ref{th:TH2}, the only task to complete for an explicit characterization of the parameter $\lambda$ appearing in the statement consists in estimating explicitly
lower bounds for $r(m,\kappa_0)$ and $r(m,-\kappa_0)$. We provide indications for only $r(m,\kappa_0)$.
By definition, every $x\in  B^{\kappa_0}_{< r(m,\kappa_0)}$ belongs to $B_{<1+m}^0$. There exists $x\in B^{\kappa_0}_{r(m,\kappa_0)}\cap B_{\leq 1+m}^0$ and then
$\vert r(m,\kappa_0)-(1+m)\vert\leq C^2_n\kappa_0^{1-(n-1)/2n}$ according to 
\eqref{eq:diff-V}. One deduces immediately an explicit lower bound for $r(m,\kappa_0)$

\end{proof}
\subsection{ISS-type of result}\label{ss:iss}
In this section, we provide the second step for our partial solution of the prescribed-time stabilization of the $n$-th order chain of integrators in presence of disturbances. More precisely,  the aim consists in stabilizing \eqref{eq:n-int-U} with a static feedback law $u=F(x)$, in a robust manner, i.e., with respect to measurement noise and external disturbances. The corresponding $n$-th order perturbed chain of integrators is given by
  \beq\label{eq:n-int-P}
   \dot x=J_nx+b(x)F(x+d_1)\ e_n+d_2,
   \eeq
   where $d_1\in\mathbb{R}^n$ is the measurement noise and $d_2\in\mathbb{R}^n$ the external perturbation. We set $d=(d_1,d_2)\in\mathbb{R}^{2n}$ and we refer to it as the perturbation. Note that we are allowing unmatched uncertainties.
 
We now provide an ISS type of result regarding the robust properties of the perturbed system \eqref{eq:n-int-P} stabilized with $k(x)=\omega_{\kappa(x)}^{H}(x)$ given by
\begin{equation}\label{eq:closedS0-P}
 \dot x=J_nx+b\omega_{\kappa(x+d_1)}^{H}(x)\ e_n+d_2, \quad  x,d_1,d_2\in\mathbb{R}^n,
 \end{equation}
 where $b$ verifies \eqref{eq:bb}. As before, we can assume with no loss of generality that $\underline{b}=1$. We have the following result.
\begin{Theorem}\label{th:TH3}
With the assumptions of Theorem~\ref{th:TH2}, System \eqref{eq:closedS0-P} is (ISS) for any bounded $d=(d_1,d_2)$.
If $d_1=0$ and $d_2$ is parallel to $e_n$ (matched uncertainty),
 then convergence occurs in fixed time. The same conclusion holds for any prescribed time $T$ by using the feedback $k_\mu(x)=\omega_{\kappa(D^\rf_{\mu} x)}^{H}(D^\rf_{\mu} x)$ with $\mu>0$ depending on $T$.
\end{Theorem}
\begin{remark}
This result improves \cite[Corollary 1]{LREPP-2018} where only the property (ISpS) was obtained.
\end{remark}
\begin{remark}
Using $k_\mu$ instead of $k_1$ will modify the gain functions in Definition~\ref{def:ISpS} since the disturbance $d=(d_1,d_2)$ must be modified to $(D^\rf_{\mu} d_1,D^\rf_{\mu} d_2/\mu)$. 
\end{remark}
For the sequel, we need the following definition. A function of class $KL$ is a continuous function $F:\mathbb{R}_+\to\mathbb{R}_+$ which is increasing, $F(0)=0$ and tends to infinity as its argument tends to infinity. 

To prove the theorem, we are not able to exhibit an ISS-Lyapunov function but, by taking into account Theorem~\ref{th:TH1} and using the characterization of (ISS) provided by \cite[Theorem 2]{Sontag-2007}, it is enough to prove the following proposition.
\begin{proposition}\label{prop:ISS}
There exists a function $F$ of class $KL$
such that for every bounded disturbances $d_1,d_2:\mathbb{R}_+\to\mathbb{R}^n$ and any trajectory of  \eqref{eq:closedS0-P}, one has
\begin{equation}\label{eq:est-ISS}
\limsup_{t\to\infty} Z(x(t))\leq F(\Vert d_1\Vert_{\infty}+\Vert d_2\Vert_{\infty}),
\end{equation}
where 
\begin{equation}\label{eq:Z}
Z(x)=\min\Big(V_0(x),V_{\kappa_0}^{1+\alpha(\kappa_0)}(x),V_{-\kappa_0}^{1-\alpha(\kappa_0)}(x)\Big), \quad x\in \mathbb{R}^n.
\end{equation}
\end{proposition}
\begin{proof} The argument is similar to Item $(S-\infty)$ in 
\cite[Proposition 2]{CHL-2015}. It is based on the following three inequalities, whose proofs are given in Appendix.
\begin{description}
\item[(i)] On the open set $B^{0}_{>1+m}$, the time derivative $\dot V_{\kappa_0}$
of $V_{\kappa_0}$ along trajectories of  \eqref{eq:closedS0-P} verifies almost everywhere
\begin{equation}\label{eq:diff-1}
\dot V_{\kappa_0}\leq -\frac{C}2V_{\kappa_0}^{1+\alpha(\kappa_0)}+F_1(\Vert d_1\Vert_{\infty}+\Vert d_2\Vert_{\infty}),
\end{equation}
where $F_1$ is a function of class $KL$.
\item[(ii)] On the  set $B^{0}_{1-m,1+m}$, the time derivative $\dot V_{0}$
of $V_{0}$ along trajectories of  \eqref{eq:closedS0-P} verifies almost everywhere
\begin{equation}\label{eq:diff-2}
\dot V_{0}\leq -\frac{C}2V_{0}+F_2(\Vert d_1\Vert_{\infty}+\Vert d_2\Vert_{\infty}),
\end{equation}
where $F_2$ is a function of class $KL$.
\item[(iii)] On the open set $B^{0}_{<1-m}$, the time derivative $\dot V_{-\kappa_0}$
of $V_{-\kappa_0}$ along non trivial trajectories of  \eqref{eq:closedS0-P} verifies almost everywhere
\begin{equation}\label{eq:diff-3}
\dot V_{-\kappa_0}\leq -\frac{C}2V_{-\kappa_0}^{1+\alpha(-\kappa_0)}+F_3(\Vert d_1\Vert_{\infty}+\Vert d_2\Vert_{\infty}),
\end{equation}
where $F_3$ is a function of class $KL$.
\end{description}
Let $x(\cdot)$ be a non trivial trajectory of  \eqref{eq:closedS0-P}.

Assuming that we have at hand the above inequalities. Suppose  first that there exists a time $t_0\geq 0$ such that one of the following situations occurs:
\begin{description}
\item[(a)] for every $t\geq t_0$, $x(t)\in B^{0}_{>1+m}$. By using \eqref{eq:diff-1}, one gets that
$$
\limsup_{t\to\infty}V_{\kappa_0}^{1+\alpha(\kappa_0)}(x(t))\leq \frac{2F_1(\Vert d_1\Vert_{\infty}+\Vert d_2\Vert_{\infty})}C;
$$
\item[(b)] for every $t\geq t_0$, $x(t)\in B^{0}_{1-m,1+m}$. By using \eqref{eq:diff-2}, one gets that
$$
\limsup_{t\to\infty}V_{0}(x(t))\leq \frac{2F_2(\Vert d_1\Vert_{\infty}+\Vert d_2\Vert_{\infty})}C;
$$
\item[(c)] for every $t\geq t_0$, $x(t)\in B^{0}_{<1-m}$. By using \eqref{eq:diff-3}, one gets that
$$
\limsup_{t\to\infty}V_{\kappa_0}^{1+\alpha(-\kappa_0)}(x(t))\leq \frac{2F_3(\Vert d_1\Vert_{\infty}+\Vert d_2\Vert_{\infty})}C.
$$
\end{description}
Let $I_+$ ($I_-$ respectively) be the set of times $t$ such that $x(t)\in B^{0}_{>1+m}$ ($x(t)\in B^{0}_{<1-m}$ respectively). 
If such a $t_0$ does not exists, either $I_+$ or $I_-$ is an infinite (countable) union of disjoint non trivial intervals $(s_k,t_k)$, $k\geq 0$, where $\lim_{k\to\infty}s_k=\infty$. We analyse only the case where $I_+=\cup_{k\geq 0}(s_k,t_k)$ since handling the other case is entirely similar. 

Set $C_V:=\max_{x\in B^0_{1+m}}V_{\kappa_0}$. 
For $k\geq 0$, consider the trajectory $x(\cdot)$ on $[t_k,s_{k+1}]$. Recall that 
$V_0(x(t_k))=V_0(x(s_{k+1}))=1+m$ by definition of $t_k,s_{k+1}$. Then, there exists 
$\tilde{t}_k\in [t_k,s_{k+1})$ such that $V_0(x(\tilde{t}_k))\leq V_0(x(s_{k+1}))$
and $V_0(x(t_))\geq 1-m$ on $[\tilde{t}_k,s_{k+1}]$. Integrating \eqref{eq:diff-2}
from  $\tilde{t}_k$ to $s_{k+1}$ yields that $\frac{C(1-m)}2\leq F_2(\Vert d_1\Vert_{\infty}+\Vert d_2\Vert_{\infty})$. Set now $L:=\limsup_{t\to\infty}V_{\kappa_0}^{1+\alpha(\kappa_0)}(x(t))$. If $L\leq C_V^{1+\alpha(\kappa_0)}$, then
$$
L\leq \frac{2C_V^{1+\alpha(\kappa_0)}}{C(1-m)}F_2(\Vert d_1\Vert_{\infty}+\Vert d_2\Vert_{\infty}).
$$
Otherwise, assume that $L> C_V^{1+\alpha(\kappa_0)}$. 
Consider then the non empty set of $v> C_V^{1+\alpha(\kappa_0)}$ for which there exist two sequences
$t_k\leq \tilde{t}_k<\tilde{s}_{k+1}<s_k$ such that 
$$
V_{\kappa_0}^{1+\alpha(\kappa_0)}(x(\tilde{t}_k))=V_{\kappa_0}^{1+\alpha(\kappa_0)}(x(\tilde{s}_{k+1}))=v \hbox{ and }V_{\kappa_0}^{1+\alpha(\kappa_0)}(x(t))\geq v,\ t\in [\tilde{t}_k,\tilde{s}_{k+1}].
$$ 
Clearly $L$ is the supremum of such $v$'s.
Integrating \eqref{eq:diff-1} between $\tilde{t}_k$ and $\tilde{s}_{k+1}$
yields at once that 
$v\leq \frac{2F_1(\Vert d_1\Vert_{\infty}+\Vert d_2\Vert_{\infty})}C$. We deduce at once that the content of Item $(b)$ above holds true. By collecting all the cases, we conclude the proof of Proposition~\ref{prop:ISS}.

\end{proof}

\section{Conclusion}\label{s:conclusion}
In this paper, we have addressed the issue of prescribed-time stabilisation of an
$n$-chain of integrators, $n\geq 1$, either pure or perturbed. We have first recasted the results obtained in  \cite{Krstic2017} within the framework of time-varying homogeneity and hence provided simpler proofs. As noticed in  \cite{Krstic2017}, the feedback laws (linear or finite time) arising from this time-varying approach do not perform well when the $n$-chain of integrators is subject to perturbations (especially measurement noise), even if one stops before the prescribed settling time. We instead propose to rely on feedback laws handling fixed-time stabilisation and to apply a standard trick of time-scale reparametrisation and homogeneity  to render the modified stabilisers fit for prescribed-time stabilisation of an $n$-th order chain perturbed of integrators. We perform that strategy in two steps. The first one That two-step strategy consists in using feedbacks similar to those of \cite{HLCH-2017} and then by relying on a nice deformation argument proposed in \cite{LREPP-2018}.  In a second step, we obtain an ISS type of result in the presence of measurement noise for prescribed-time stabilisation of an $n$-th perturbed  chain of integrators.
However, such an approach is meaningful if one can get  an explicit hold on the various parameters involved in the above construction. This is why we devoted a section for such an objective.

\section{Appendix}\label{s:app}
\subsection{ Proof of Proposition~\ref{prop:LMI0}}
We next prove the result for $\eta=1$ and the argument is inspired from the proof of Lemma 4.0 of \cite{GK-1994}, and partly given in \cite{CS-2010}.
Given a vector $K=(k_1,\cdots,k_n)^T\in\mathbb{R}^n$ with positive entries, we consider the invertible $n\times n$  matrix $M_K$ defined by
\begin{equation}\label{eq:MK}
\begin{aligned}
M_K & = \begin{pmatrix}
k_1&k_2&\cdots&k_n\\
0&k_1&\cdots&k_{n-1}\\
\vdots & \vdots & \ddots &\vdots \\
0&0 &\cdots&k_1\\
\end{pmatrix}.
\end{aligned}
\end{equation}
Note that
$$
M_K e_n=K,\ M^T_Ke_1=K, \ M_KJ_nM^{-1}_K=J_n.
$$
The last equation comes from the fact that $M_K$ is a polynomial function of $J_n$, namely $M_K=\sum_{i=1}^nk_iJ^{i-1}_n$. 

Multiplying the LMI \eqref{eq:LMI0} on the left and on the right by $(M^T_K)^{-1}$ and $M^{-1}_K$ respectively yields the following LMI
$$
\Big(J_n-b\ Ke_1^T\Big)^TS_1+S_1\Big(J_n-b Ke_1^T\Big)\leq -\rho M^T_KM_K,\quad \forall b\geq \underline{b}.
$$
where $S_1=(M^T_K)^{-1}SM^{-1}_K$.  Let $\rho_*>0$ such that $\rho M^T_KM_K\geq \rho_* Id_n$.

We are left to prove that there exists $\rho_*>0$, $S_1$ symmetric positive definite and a vector $K\in\mathbb{R}^n$ so that the following LMI holds true,  
\beq\label{eq:LMI-1}
\Big(J_n-b\ Ke_1^T\Big)^TS_1+S_1\Big(J_n-b\ Ke_1^T\Big)\leq -\rho_* Id_n,\quad \forall b\geq \underline{b}.
\eeq
For $n=1$, \eqref{eq:LMI-1} reduces to $-2bkS\leq -\mu_*$. By taking $S=1/2$ and $k=1/\underline{b}$ we get the result with $\rho_*=1$.

Let $n$ be a positive integer larger than or equal to two. Set $\tilde{e}_1=(1,\cdots,0)^T\in\mathbb{R}^{n-1}$ and $K=(k_1,L^T)^T$ with $L\in\mathbb{R}^{n-1}$ to be determined.
Notice that 
$$
J_n-b\ Ke_1^T=\begin{pmatrix}-b\ k_1&\tilde{e}_1^T\\-b\ L&J_{n-1}\\\end{pmatrix}.
$$
For $\Omega\in\mathbb{R}^{n-1}$, consider the $n\times n$ matrix $A_\Omega$ given by 
$$
A_\Omega=\begin{pmatrix}1&0\\ \Omega&Id_{n-1}\\ 
\end{pmatrix},
$$
We make the linear change of variable $y=A_\Omega x$ and we require the following condition on
$(k_1,L)$, i.e., $k_1\Omega+L=0$. One gets that 
$$
A_\Omega(J_n-b\ Ke_1^T)A^{-1}_\Omega=
\begin{pmatrix}-(b\ k_1+\tilde{e}_1^T\Omega)&\tilde{e}_1^T\\
-(J_{n-1}+\Omega\tilde{e}_1^T)\Omega&J_{n-1}+\Omega\tilde{e}_1^T\\
\end{pmatrix}.
$$
This linear change of variable amounts to multiply \eqref{eq:LMI-1} on the left by $(A_\Omega^T)^{-1}$ and on the right by $A^{-1}_\Omega$ and we still denote by $S$ the matrix $(A_\Omega^T)^{-1}SA^{-1}_\Omega$. We now pick $\Omega$ so that $J_{n-1}+\Omega\tilde{e}_1^T$ is Hurwitz and there exists a positive constant $\mu>0$ and a real symmetric positive definite $(n-1)\times (n-1)$ matrix $S_{n-1}>0$ such that
 $$
 (J_{n-1}+\Omega\tilde{e}_1^T)^TS_{n-1}+S_{n-1}(J_{n-1}+\Omega\tilde{e}_1^T)-\leq \rho_* Id_{n-1}.
 $$
After choosing $S=\begin{pmatrix}1&0\\
0&S_{n-1}\\
\end{pmatrix}$, one simply finds $k_1>0$ large enough to get the result. 

\begin{remark} One must notice the similarity of the argument which is essentially that of \cite{GK-1994} and \cite{CS-2010}, with the corresponding one in \cite{Krstic2017}. The one given here is more transparent and also allows to use the extra degree of freedom given by $\eta$. 
\end{remark}

\subsection{Proof of Equations \eqref{eq:diff-1}, \eqref{eq:diff-2} and \eqref{eq:diff-3}}
For $\kappa\in \{-\kappa_0,0,\kappa_0\}$, taking the time derivative $\dot V_{\kappa}$ of $V_{\kappa_0}$ along a trajectory of 
 \eqref{eq:closedS0-P} yields the inequality
 $$
 \dot V_{\kappa}\leq -C V_{\kappa_0}^{1+\alpha(\kappa)}(x)+b\langle \nabla V_{\kappa}(x),e_n\rangle\Big(\omega_{\kappa(x+d_1)}^{H}(x)-\omega_{\kappa(x)}^{H}(x)\Big)
+\langle \nabla V_{\kappa}(x),d_2\rangle.
$$
We will prove that in each region of interest, there exists $KL$ functions $F_1,F_2$ such that
\begin{equation}\label{eq:est11}
\overline{b}\Big\vert \langle \nabla V_{\kappa}(x),e_n\rangle\Big(\omega_{\kappa(x+d_1)}^{H}(x)-\omega_{\kappa(x)}^{H}(x)\Big)\Big\vert\leq \frac{C}4V_{\kappa}^{1+\alpha(\kappa)}(x)+F_1(\Vert d_1\Vert_{\infty}),
\end{equation}
and
\begin{equation}\label{eq:est12}
\Big\vert \langle \nabla V_{\kappa}(x),d_2\rangle\Big\vert\leq \frac{C}4V_{\kappa_0}^{1+\alpha(\kappa)}(x)+
F_2(\Vert d_2\Vert_{\infty}).
\end{equation}
Once this is established, one gets the conclusion by taking $F=F_1+F_2$.

We start by proving \eqref{eq:est12}. For $\kappa\in \{-\kappa_0,0\}$, the region of interest is bounded. Hence one immediately concludes by applying Cauchy-Schwartz inequality and taking an upper bound for the continuous function $\Vert \nabla V_{\kappa}\Vert$ on the region of interest. For $\kappa=\kappa_0$, we recall that, for $1\leq i\leq n$,  $\langle \nabla V_{\kappa}(x),e_i\rangle$ is $\mathbf{r}(\kappa_0)$-homogeneous of degree $(2+\kappa_0)-r_i$ with respect to the family of dilations $\big(D^{\mathbf{r}(\kappa_0)}_{\lambda}\big)_{\lambda>0}$. It is therefore immediate to see that there exists a positive constant $C_i$ such that $\vert \langle \nabla V_{\kappa_0}(x),e_i\rangle\vert \leq C_iV_{\kappa_0}^{1-\frac{r_i}{2+\kappa_0}}$ over $\mathbb{R}^n$. One deduces that
\begin{equation}\label{eq:VnV}
\Big\vert \langle \nabla V_{\kappa_0}(x),d_2\rangle\Big\vert\leq \sum_{i=1}^nC_iV_{\kappa_0}^{1-\frac{r_i}{2+\kappa_0}}\vert (d_2)_i\vert.
\end{equation}
Since every $r_i$ is positif and hence $1-\frac{r_i}{2+\kappa_0}<1+\alpha(\kappa_0)$, one can apply an appropriately weighted Holder inequality to get \eqref{eq:est12}.

We know turn to an argument for \eqref{eq:est11}. We provide an argument only for $\kappa=\kappa_0$ since for the other cases it is similar. If $\kappa(x+d_1)\neq \kappa_0$, then 
$V_0(x+d_1)\leq 1+m$ implying that $\Vert x\Vert\leq C_1\Vert d_1\Vert$ for some positive constant independent of $x,d_1$. Hence one can bound the left-hand side of \eqref{eq:est11} by $F_1(\Vert d_1\Vert)$ for some $KL$ function $F_1$, and then conclude. We now treat the case where $\kappa(x+d_1)= \kappa_0$. Recall that 
$\omega_{\kappa_0}^{H}$ is $\mathbf{r}(\kappa_0)$-homogeneous of degree $r_{n+1}:=1+n\kappa_0$ with respect to the family of dilations $\big(D^{\mathbf{r}(\kappa_0)}_{\lambda}\big)_{\lambda>0}$. For non zero $x\in \mathbb{R}^n$, we define the normalized vector 
$$
[x]^{\kappa_0}:=\frac{x}{V_{\kappa_0}^{\frac1{2+\kappa_0}}}\in B^{\kappa_0}_1.
$$
Then one has on $B^{\kappa_0}_{>1+m}$,
\begin{eqnarray}
&\langle \nabla V_{\kappa_0}(x),e_n\rangle\Big(\omega_{\kappa_0}^{H}(x+d_1)-\omega_{\kappa_0}^{H}(x)\Big)=\nonumber\\
&V_{\kappa_0}^{1-\frac{r_n}{2+\kappa_0}}(x)
\langle \nabla V_{\kappa_0}([x]^{\kappa_0}),e_n\rangle
\Big(V_{\kappa_0}^{\frac{r_{n+1}}{2+\kappa_0}}(x+d_1)\omega_{\kappa_0}^{H}([x+d_1]^{\kappa_0})-V_{\kappa_0}^{\frac{r_{n+1}}{2+\kappa_0}}(x)\omega_{\kappa_0}^{H}([x]^{\kappa_0})\Big)\nonumber\label{eq:final}\\
&=V_{\kappa_0}^{1+\alpha(\kappa_0)}(x)
\langle \nabla V_{\kappa_0}([x]^{\kappa_0}),e_n\rangle M(x,d_1),\label{eq:final}
\end{eqnarray}
where 
$$
M(x,d_1):=
\Big(\frac{V_{\kappa_0}(x+d_1)}{V_{\kappa_0}(x)}\Big)^{\frac{r_{n+1}}{2+\kappa_0}}\omega_{\kappa_0}^{H}([x+d_1]^{\kappa_0})-
\omega_{\kappa_0}^{H}([x]^{\kappa_0}).
$$
Moreover, we have the following result: there exists a positive constant $B$ such that, for every $x,d\in\mathbb{R}^n$ with $\Vert x\Vert,\Vert d\Vert\leq 1$, one has 
\begin{equation}\label{eq:est-om}
\vert \omega_{\kappa_0}^{H}(x+d)-\omega_{\kappa_0}^{H}(x)\vert \leq B \Vert d\Vert^{r_{n+1}},
\end{equation}
which is an immediate consequence of \eqref{eq:elem3}. 

Consider $\rho>0$ to be fixed small later. Assume first that 
$$
\frac{V_{\kappa_0}(d_1)}{V_{\kappa_0}(x)}=V_{\kappa_0}(\frac{d_1}{V_{\kappa_0}^{\frac1{2+\kappa_0}}})\leq \rho.
$$
We rewrite $M(x,d_1)$ as 
\begin{equation}\label{eq:M}
M(x,d_1)=\Big[\Big(\frac{V_{\kappa_0}(x+d_1)}{V_{\kappa_0}(x)}\Big)^{\frac{r_{n+1}}{2+\kappa_0}}-1\Big]\omega_{\kappa_0}^{H}([x+d_1]^{\kappa_0})+
\Big(\omega_{\kappa_0}^{H}([x+d_1]^{\kappa_0})-\omega_{\kappa_0}^{H}([x]^{\kappa_0})\Big).
\end{equation}
The term in brackets in \eqref{eq:M} can be written as
$$
\Big(\frac{V_{\kappa_0}(x+d_1)}{V_{\kappa_0}(x)}\Big)^{\frac{r_{n+1}}{2+\kappa_0}}-1=
V_{\kappa_0}([x]^{\kappa_0}+\frac{d_1}{V_{\kappa_0}^{\frac1{2+\kappa_0}}})-
V_{\kappa_0}([x]^{\kappa_0}),
$$
Notice that $[x]^{\kappa_0}$ and $\frac{d_1}{V_{\kappa_0}^{\frac1{2+\kappa_0}}}$ 
belong to the compact set $B^{\kappa_0}_{\leq 1+m}$ and hence, since $V_{\kappa_0}$ is of class $C^1$, there exists a positive constant $C_2$ independent of $x,d_1$, such that 
$$
\Big\vert V_{\kappa_0}([x]^{\kappa_0}+\frac{d_1}{V_{\kappa_0}^{\frac1{2+\kappa_0}}})-
V_{\kappa_0}([x]^{\kappa_0})\Big\vert\leq C_2 V_{\kappa_0}(\frac{d_1}{V_{\kappa_0}^{\frac1{2+\kappa_0}}})^{\frac1{2+\kappa_0}}\leq C_2\rho^{\frac1{2+\kappa_0}}.
$$
By using \eqref{eq:est-om}, we can bound the term in parentheses in \eqref{eq:M} as follows,
$$
\vert \omega_{\kappa_0}^{H}([x+d_1]^{\kappa_0})-\omega_{\kappa_0}^{H}([x]^{\kappa_0})\vert\leq B \Vert [x+d_1]^{\kappa_0}-[x]^{\kappa_0}\Vert^{r_{n+1}}.
$$
In turn, one has
$$
[x+d_1]^{\kappa_0}-[x]^{\kappa_0}=\Big[\Big(\frac{V_{\kappa_0}(x)}{V_{\kappa_0}(x+d_1)}\Big)^{\frac{r_{n+1}}{2+\kappa_0}}-1\Big][x]^{\kappa_0}+
\Big(\frac{V_{\kappa_0}(x+d_1)}{V_{\kappa_0}(x)}\Big)^{\frac{r_{n+1}}{2+\kappa_0}}\frac{d_1}{V_{\kappa_0}^{\frac1{2+\kappa_0}}}.
$$
Using the homogeneity property of $V_{\kappa_0}$, one gets that there exists a positive constant $C_3$ independent of $x,d_1$ such that 
$M(x,d_1)\leq C_3\rho^{\frac1{2+\kappa_0}}$. Since $\omega_{\kappa_0}^H$ is bounded on $B^0_1$, one gets 
$$
\Big\vert \langle \nabla V_{\kappa}(x),e_n\rangle\Big(\omega_{\kappa(x+d_1)}^{H}(x)-\omega_{\kappa(x)}^{H}(x)\Big)\Big\vert\leq \frac{C}4V_{\kappa_0}^{1+\alpha(\kappa)}(x),
$$
for $\rho$ small enough, and hence \eqref{eq:est11}.

We now assume that 
\begin{equation}\label{eq:ouf}
\frac{V_{\kappa_0}(d_1)}{V_{\kappa_0}(x)}
>\rho.
\end{equation}
In that case, the conclusion follows if one can prove that there exists $C_4>0$ independent of $x,d_1$ such that 
\begin{equation}\label{eq:final1}
V_{\kappa_0}(x+d_1)\leq C_4V_{\kappa_0}(d_1)+F_1(\Vert d_1\Vert),
\end{equation}
for some $KL$ function $F_1$. 
Indeed, in \eqref{eq:final}, the term in parentheses becomes bounded by $F_2(\Vert d_1\Vert)$
for some $KL$ function $F_2$ and then one gets \eqref{eq:est11} after using Holder's inequality with appropriate weights.

We are then left to prove \eqref{eq:final1}. For that purpose set $f(s):=V_{\kappa_0}(x+sd_1)$ for $s\in[0,1]$ and let $s^*\in \in[0,1]$ such that $f(s^*)=\max_{s\in [0,1]}f(s)$.
We will prove \eqref{eq:final1} with $f(s^*)$ on the left-hand side and hence the conclusion. We can therefore assume with no loss of generality that $s^*=1$.
One has 
\begin{eqnarray*}
f(1)-f(0)&\leq&\int_0^1\vert f'(s)\vert ds=\int_0^1\vert\langle \nabla V_{\kappa_0}(x+sd_1),d_1\rangle\vert ds\\
&\leq&\sum_{i=1}^n \int_0^1\Big\vert\frac{\partial V_{\kappa_0}}{\partial x_i}(x+sd_1)(d_1)_i\Big\vert ds\leq C \sum_{i=1}^n \int_0^1V_{\kappa_0}^{\frac{2+\kappa_0-r_i}{2+\kappa_0}}(x+sd_1)\vert (d_1)_i\vert ds\\
 &\leq&C\sum_{i=1}^nV_{\kappa_0}^{\frac{2+\kappa_0-r_i}{2+\kappa_0}}(x+d_1)\vert (d_1)_i\vert\leq \frac{f(1)}2+CF_3(\Vert d\Vert),
\end{eqnarray*}
for some $KL$ function $F_3$.  In the above, we have used \eqref{eq:VnV}, the definition of $s^*=1$ and for the final inequality, Holder's inequality with appropriate weights. Combining the previous inequality with \eqref{eq:ouf}, one concludes the argument for \eqref{eq:final1} and hence that of 
 \eqref{eq:final1}.
\bibliography{biblio-YR}
\bibliographystyle{abbrv}
\end{document}